\documentclass[a4paper,10pt]{article}

\usepackage{graphicx}
\usepackage{amssymb}
\usepackage[all]{xy}
\usepackage{graphicx}
\usepackage{pb-diagram}
\usepackage{amsthm}
\usepackage{amsmath}
\usepackage{amsmath}
\usepackage{amsfonts}
\usepackage{mathrsfs}
\usepackage{amscd}
\usepackage{pb-diagram}
\usepackage{color}
\usepackage[all]{xy}
\usepackage{url}

\author{Mathieu Molitor
\\ \it\small{Chaire des structures alg\'ebriques et g\'eom\'etriques}
\\ \it\small{Facult\'e des sciences de base -- Institut de math\'ematiques B}
\\ \it\small{Ecole Polytechnique F\'ed\'erale de Lausanne}
\\ \small{\it{e-mail:}}\,\,\url{pergame.mathieu@gmail.com}
}

\title{The group of unimodular automorphisms of a principal bundle and 
the Euler-Yang-Mills equations}

\date{}

\begin{document}

\newtheorem{lemma}{Lemma}[section]
\newtheorem{definition}[lemma]{Definition}
\newtheorem{proposition}[lemma]{Proposition}
\newtheorem{corollary}[lemma]{Corollary}
\newtheorem{theoreme}[lemma]{Theorem}
\newtheorem{remarque}[lemma]{Remark}
\newtheorem{example}[lemma]{Example}

\maketitle

\begin{abstract}
	Given a principal bundle $G\hookrightarrow P\rightarrow B$ (each being compact, 
	connected and oriented) and a $G$-invariant metric $h^{P}$ on $P$ which induces 
	a volume form $\mu^{P}\,,$ we consider the group of all unimodular automorphisms
	 $\textup{SAut}(P,\mu^{P}):=\{\varphi\in\textup{Diff}(P)\,|\,\varphi^{*}\mu^{P}=
	\mu^{P}\,\,\text{and}\,\,\varphi\,\,\text{is}\,\,G\text{-equivariant}\}$
	of $P\,,$ and determines its Euler equation \`a la Arnold. The resulting equations 
	turn out to be (a particular case of) the Euler-Yang-Mills equations of an 
	incompressible classical charged ideal fluid moving on $B\,.$\\
	It is also shown that the group $\textup{SAut}(P,\mu^{P})$ is an extension of a 
	certain volume preserving diffeomorphisms group of $B$ by the gauge group 
	$\textup{Gau}(P)$ of $P\,.$
\end{abstract}


%

\section{Introduction}
	Since \cite{Arnold}, it is well known that an appropriate configuration 
	space for the study of equations of hydrodynamical type (more precisely, 
	the incompressible Euler equations of an incompressible fluid) on a 
	Riemannian manifold $(M,g)$ endowed with a volume form $\mu$ 
	($\mu$ being not necessarily induced by the metric $g$), is given 
	by the group of all unimodular diffeomorphisms $\text{SDiff}(M,\mu):=
	\{\varphi\in\text{Diff}(M)\,|\,\varphi^{*}\mu=\mu\}$ of $M\,.$ 
	This group is --in a suitably chosen sense-- an infinite dimensional 
	Lie group whose Lie algebra $\mathfrak{X}(M,\mu):=\{X\in 
	\mathfrak{X}(M)\,\vert\,\text{div}_{\mu}\,(X)=0\}$ 
	is the space of divergence free vector fields endowed with the opposite of the 
	usual vector field bracket, and if $X\in \mathfrak{X}(M,\mu)$ is 
	a time-dependant divergence free vector field describing the velocity field of an
	incompressible fluid, then 
	its dynamics is governed by the incompressible Euler equation 
	$\frac{d}{dt}X+\nabla_{X}X=\nabla p\,,$ where $p$ is the pressure of the fluid.
	It turns out that this equation characterizes 
	geodesics on $\text{SDiff}\,(M,\mu)$ with respect to the natural 
	right-invariant $L^{2}$-metric on $\text{SDiff}\,(M,\mu)$ (see 
	\cite{Ebin-Marsden}), and can be seen as an Euler equation 
	(or Lie-Poisson equation) on the ``regular dual" of $\mathfrak{X}(M,\mu)$ 
	(see \cite{Arnold-Khesin}).\\
	In this paper, we propose another configuration space to study 
	the Euler equation when some symmetries are involved. Our point of 
	departure is to assume that the fluid evolves on the total space 
	of a principal bundle $G\hookrightarrow P\rightarrow B$ ($P$ 
	being connected and oriented). We assume also that the metric $h^{P}$ 
	on $P$ is $G$-invariant. In particular, the volume form $\mu^{P}$ on 
	$P$ induced by $h^{P}$ is also $G$-invariant. This leads naturally 
	to consider the group $\text{SAut}(P,\mu^{P})$ of automorphims of $P$ 
	preserving the volume form $\mu^{P}$ instead of the group 
	$\text{SDiff}(P,\mu^{P})\,.$ In other words, we assume the vector field 
	describing the velocity of the fluid to be initially $G$-invariant. 
	This approach allows us to describe the Euler equation (in the presence 
	of symmetries), as a system of two coupled equations, one living on the 
	space of free divergence (for a certain volume form) vector fields on 
	$B\,,$ the other living on the Lie algebra of the gauge group $\text{Gau}(P)$ 
	of $P\,.$ In some cases, these equations are a particular case of 
	the Euler-Yang-Mills equation of an incompressible classical charged 
	ideal fluid moving on $B\,,$ and are physically relevant for the cases 
	$G=S^{1}$ (super-conductivity equation, see \cite{Vizman}), $G=SU(2)$ 
	and $G=SU(3)$ (chromohydrodynamics, see \cite{Gibbons,Gay-Balmaz-Ratiu})\,. 
	The terminology ``Euler-Yang-Mills equation'' comes from \cite{Gay-Balmaz-Ratiu}\,.\\\\
	The second section of this paper describes the Lie group structure 
	of the group $\text{SDiff}(M,\mu)^{G}$ of all $G$-equivariant 
	diffeomorphisms of a compact manifold $M$ which preserve a volume 
	form $\mu\,.$ The arguments are essentially those used by Hamilton 
	in \cite{Hamilton}, Theorem 2.5.3, except that one has to check the 
	constructions involving the Nash-Moser inverse function theorem to 
	``respect symmetries". In section \ref{chapitre 2 section 2}, 
	the careful study of the ``structure" of a $G$-invariant volume 
	form $\mu^{P}$ on the total space $P$ of a principal bundle 
	$G\hookrightarrow P\longrightarrow B\,,$ allows us to give an 
	integration formula (Proposition \ref{proposition vraiment formule integration}) 
	which is necessary to determine the Euler equation of the group 
	$\text{SAut}(P,\mu^{P})$ (Theorem \ref{theorem equation d'euler})\,. 
	Finally in section \ref{chapitre 2 partie 4}, we show, in the same 
	spirit of \cite{Michor}, that $\text{SAut}(P,\mu^{P})$ is a 
	$\text{Gau}(P)$-principal bundle whose base is a collection of 
	connected components of $\text{SDiff}(B,V\mu^{B})\,,$ where 
	$V\mu^{B}$ is a volume form on $B$ related to the volume of 
	the orbits of $P\,.$ In particular, $\text{SAut}(P,\mu^{P})$ 
	is a non-abelian extension of this collection of connected 
	components of $\text{SDiff}(B,V\mu^{B})$ by the gauge group $\text{Gau}(P)\,.$

\section{The group $\text{SDiff}\,(M,\mu)^{G}$ as a tame Lie group}

	This section 
	deals with the differentiable and Lie group structure of some subgroups 
	of the group of smooth diffeomorphisms of a compact manifold, 
	using the infinite dimensional geometry point of view. 
	For that purpose, we will use the category of tame Fr\'echet manifolds 
	developed by Hamilton in \cite{Hamilton}, and not simply the usual category of 
	Fr\'echet manifolds\footnote{
	The reader should be aware that beyond the Banach case, several nonequivalent theories of 
	infinite dimensional manifolds coexist (see \cite{Keller}), 
	but when the modelling spaces are Fr\'echet spaces, then 
	most of these theories coincide, 
	and it is thus natural to talk, without 
	any further references, of a Fr\'echet manifold (as defined in \cite{Hamilton} 
	for example).}. This choice is motivated by the necessity to 
	use an inverse function theorem, which is available in Hamitlon's category 
	contrary to the general Fr\'echet setting.\\		
	For the convenience of the reader, we recall here the basic definitions relevant 
	for Hamilton's category : 

\begin{definition}
	\begin{description}
		\item[$(i)$] A graded Fr\'echet space $(F,\{\|\,.\,\|_{n}\}_{n\in\mathbb{N}})\,,$
			is a Fr\'echet space $F$ whose topology is defined by a collection of seminorms 
			$\{\|\,.\,\|_{n}\}_{n\in\mathbb{N}}$ which are increasing in strength:
			\begin{eqnarray}
			\|x\|_{0}\leq\|x\|_{1}\leq\|x\|_{2}\leq \cdots
			\end{eqnarray}
			for all $x\in F\,.$
		\item[$(ii)$] A linear map $L\,:\,F\rightarrow G$ between two graded Fr\'echet spaces
			$F$ and $G$ is tame (of degree $r$ and base $b$) if for all $n\geq b\,,$
			there exists a constant $C_{n}>0$ such that for all $x\in F\,,$
			\begin{eqnarray}
				\|L(x)\|_{n}\leq C_{n}\,\|x\|_{n+r}\,.
			\end{eqnarray}
		\item[$(iii)$] If $(B,\|\,.\,\|_{B})$ is a Banach space, then $\Sigma(B)$ denotes 
			the graded Fr\'echet space of all sequences $\{x_{k}\}_{k\in\mathbb{N}}$ of 
			$B$ such that for all $n\geq 0,$ 
			\begin{eqnarray}
				\|\{x_{k}\}_{k\in\mathbb{N}}\|_{n}:=\Sigma_{k=0}^{\infty}\,e^{nk}\|x_{k}\|_{B}
				<\infty\,.
			\end{eqnarray}
		\item[$(iv)$] A graded Fr\'echet space $F$ is tame if there exist a Banach space 
			$B$ and two tame linear maps $i\,:\,F\rightarrow \Sigma(B)$ and 
			$p\,:\,\Sigma(B)\rightarrow F$ such that $p\circ i$ is the identity on $F\,.$
		\item[$(v)$] Let $F,G$ be two tame Fr\'echet spaces, $U$ an open subset of 
			$F$ and $f\,:\,U\rightarrow G$ a map. We say that $f$ is a smooth tame map
			if $f$ is smooth\footnote{By smooth we mean that $f\,:\,U\subseteq F\rightarrow G$ 
			is continuous and that 
			for all $k\in\mathbb{N}\,,$ the $k$th derivative 
			$D^{k}f\,:\,U\times F\times \cdots \times F
			\rightarrow G$ exists and is jointly continuous on the product  space, such as 
			described in \cite{Hamilton}.} 
			and if for every $k\in\mathbb{N}$ and for every 
			$(x,u_{1},...,u_{k})\in U\times F\times \cdots F\,,$ there exist a neighborhood
			$V$ of $(x,u_{1},...,u_{k})$ in $U\times F\times \cdots F$ and 
			$b_{k},r_{0},...,r_{k}\in\mathbb{N}$ such that for every $n\geq b_{k}\,,$ there 
			exists $C_{k,n}^{V}>0$ such that 
			\begin{eqnarray}
				\|D^{k}f(y)\{v_{1},...,v_{k}\}\|_{n}\leq C_{k,n}^{V}\,\big(1+\|y\|_{n+r_{0}}
				+\|v_{1}\|_{n+r_{1}}+\cdots+\|v_{k}\|_{n+r_{k}}\big)\,,
			\end{eqnarray}
			for every $(y,v_{1},...,v_{k})\in V\,,$ where $D^{k}f\,:\,
			U\times F\times\cdots\times F\rightarrow G$ denotes the $k$th derivative of 
			$f\,.$
	\end{description}
\end{definition}
\begin{remarque}
	In the sequel, we will use interchangeably the notation 
	$(Df)(x)\{v\}$ or $f_{*_{x}}v$ for the first derivative of $f$ at a 
	point $x$ in direction $v\,.$
\end{remarque}
	As one may notice, tame Fr\'echet spaces and smooth tame maps form 
	a category, and it is thus natural to define a tame Fr\'echet manifold 
	as a Hausdorff topological space with an atlas of coordinates charts taking their value in
	tame Fr\'echet spaces, such that the coordinate transition functions are all
	smooth tame maps (see \cite{Hamilton}).
	The definition of a tame smooth map between tame Fr\'echet manifolds is then 
	straightforward, and we thus obtain a subcategory of the category of Fr\'echet manifolds.\\
	In order to avoid confusion, let us also precise our notion of submanifold. We will say 
	that a subset $\mathcal{M}$ of a 
	tame Fr\'echet manifold $\mathcal{N}\,,$ endowed with the trace topology, 
	is a submanifold, 
	if for every point $x\in \mathcal{M}\,,$ there exists a chart 
	$(\mathcal{U},\varphi)$ of $\mathcal{N}$ such that $x\in \mathcal{U}$ and such that 
	$\varphi(\mathcal{U}\cap \mathcal{M})=U\times \{0\}\,,$ where 
	$\varphi(\mathcal{U})=U\times V$ is a product of two open subsets of tame Fr\'echet spaces. 
	Note that a submanifold of a tame Fr\'echet manifold is also a tame Fr\'echet manifold.\\
	Finally, we define a tame Lie group $\mathcal{G}$ as a tame Fr\'echet manifold
	with a group structure such that the multiplication map 
	$\mathcal{G}\times \mathcal{G}\rightarrow \mathcal{G},\,(g,h)\mapsto gh$  and the inverse map
	$\mathcal{G}\rightarrow \mathcal{G},\,g\mapsto g^{-1}$
	are smooth tame maps. A tame Lie subgroup is defined as being a subset of 
	a tame Lie group which is a submanifold and a subgroup. A tame Lie subgroup 
	is in particular a tame Lie group.
\begin{remarque}
	The above notions of submanifolds, Lie groups and Lie subgroups are stated 
	in the framework of tame Fr\'echet manifolds, but of course, similar 
	definitions --that we adopt--  hold in the more general framework of Fr\'echet manifolds.
\end{remarque}
	For the sake of completeness, let us state here the raison d'\^{etre} of 
	tame Fr\'echet spaces and tame Fr\'echet manifolds (see \cite{Hamilton}) :
\begin{theoreme}[Nash-Moser inverse function Theorem]
	Let $F,G$ be two tame Fr\'echet spaces, $U$ an open subset of 
	$F$ and $f\,:\, U\rightarrow G$ a smooth tame map. If there exists 
	an open subset $V\subseteq U$ such that 
	\begin{description}
		\item[$(i)$] $Df(x)\,:\,F\rightarrow G$ is an linear isomorphism for all 
			$x\in V\,,$
		\item[$(ii)$] the map $V\times G\rightarrow F,\,(x,v)\mapsto 
			\big(Df(x)\big)^{-1}\{v\}$ is a smooth tame map, 
	\end{description}
	then $f$ is locally invertible on $V$ and each local inverse is a
	smooth tame map. 
\end{theoreme}

\begin{remarque}
	The Nash-Moser inverse function Theorem is important in geometric hydrodynamics, 
	since one of its most important geometric object, namely the group 
	of all smooth volume preserving diffeomorphims 
	$\textup{SDiff}(M,\mu):=\{\varphi\in\textup{Diff}(M)\,\vert\,\varphi^{*}\mu=\mu\}$ 
	of an oriented manifold $(M,\mu)\,,$ 
	can only be given a rigorous Fr\'echet Lie group structure by
	using an inverse function theorem (at least up to now).
	To our knowledge, only two authors succeeded in doing this. The first was Omori 
	who showed and used an inverse function theorem in terms of 
	ILB-spaces (``inverse limit of Banach spaces", see \cite{Omori}), and later on, Hamilton with 
	his category of tame Fr\'echet spaces together with the Nash-Moser inverse function 
	Theorem (see \cite{Hamilton}). Nowadays, it is nevertheless not uncommon 
	to find  mistakes or 
	big gaps in the literature when it comes to the differentiable 
	structure of $\textup{SDiff}(M,\mu)\,,$
	even in some specialized textbooks in infinite dimensional geometry. The case 
	of $M$ being non-compact is even worse, and of course, no proof that $\textup{SDiff}(M,\mu)$ 
	is a ``Lie group" is available in this case. 
\end{remarque}

	Now let $M$ be a compact manifold and $G$ a compact and connected 
	Lie group acting on $M\,.$ The action of $G$ is denoted by 
	$\vartheta\,:\,G\times M\rightarrow M$ and for $g\in G$, we write 
	$\vartheta_{g}\,:\,M\rightarrow M\,,\,\,x\mapsto \vartheta(g,x)\,.$

\begin{proposition}\label{diffG}
	The group $\text{Diff}\,(M)^{G}:=\{\varphi\in\text{Diff}\,(M)\,\vert\,
	\vartheta_{g}\circ\varphi=\varphi\circ\vartheta_{g},\,\forall 
	g\in G\}$ is a tame Lie subgroup of the group $\text{Diff}\,(M)\,.$ 
	Its Lie algebra is the space $\mathfrak{X}(M)^{G}:=
	\{X\in \mathfrak{X}(M)\,\vert\,\vartheta_{g_{*}}X=X\,,
	\forall g\in G\}\,.$
	\end{proposition}
\begin{proof} 
	Choose a $G$-invariant metric $h$ on $M$ 
	and define a map $pr\,:\,\Omega^{1}(M)\rightarrow 
	\Omega^{1}(M),\,\theta\mapsto\theta^{G}$ by
	$$
	\theta^{G}_{x}(X_{x}):=\dfrac{1}{\text{Vol}(G)}\,\int_{G}\,
	(\vartheta_{g}^{*}\theta)_{x}(X_{x})\,\nu^{G}\,,
	$$
	where $X_{x}\in T_{x}M\,.$ Since $pr$ is a continuous 
	projection, we have the following topological direct sum :\\
	$$
	\Omega^{1}(M)=\Omega^{1}(M)^{G}\oplus\text{ker}(pr),
	$$
	and as $h$ is $G$-invariant,
	\begin{eqnarray}
		\mathfrak{X}(M)=\mathfrak{X}(M)^{G}\oplus\text{ker}
		(\widetilde{pr}),\label{decomposition}
	\end{eqnarray}
	where $\widetilde{pr}\,:\,\mathfrak{X}(M)\rightarrow\mathfrak{X}(M)^{G}$ 
	is the projection obtained from $pr$ using the duality between $TM$ 
	and $T^{\overset{*}{\text{}}}M$ via the metric $h\,.$ Notice that 
	the decomposition \eqref{decomposition} implies that $\mathfrak{X}(M)^{G}$ 
	is a tame Fr\'echet space (it's a Fr\'echet space because 
	$\mathfrak{X}(M)^{G}$ is closed in $\mathfrak{X}(M)$ and it's also a 
	tame space because $\mathfrak{X}(M)$ is tame, see \cite{Hamilton}, 
	Definition 1.3.1 and Corollary 1.3.9).\\\\
	Let $(\mathcal{U},\varphi)$ be the ``standard'' chart of $\text{Diff}\,(M)$ 
	at the identity element $Id_{M}$ obtained using the metric $h\,,$ i.e., 
	$\varphi(\mathcal{U})\subseteq \mathfrak{X}(M)$ and $\varphi^{-1}(X)(x)=
	\text{exp}_{x}\,(X_{x})$ for $X\in \varphi(\mathcal{U})\subseteq 
	\mathfrak{X}(M)$ and $x\in M\,.$\\
	Restricting $\mathcal{U}$ if necessary, we may assume $\varphi(\mathcal{U})=
	U_{1}\times U_{2}$ where $U_{1}$ is an open subset of $\mathfrak{X}(M)^{G}$ 
	and $U_{2}$ an open subset of $\text{ker}(\widetilde{pr})\,.$ 
	From the $G$-invariance of $h\,,$ we also have :
	\begin{description}
		\item[$(i)$]  if $X\in U_{1}\,,$ then $\varphi^{-1}(X)\in 
			\text{Diff}\,(M)^{G}\,,$
		\item[$(ii)$]  $\text{exp}_{\vartheta_{g}(x)}\,
			(\vartheta_{g})_{*_{x}}X_{x}=\vartheta_{g}\,\big(\text{exp}_{x}\,
			(X_{x})\big)$ for all $x\in M\,,$ for all $X_{x}\in T_{x}M$ and for all $g\in G\,.$
	 \end{description}
	From (i)\, we get $\varphi^{-1}(U_{1}\times\{0\})\subseteq
	\mathcal{U}\cap\text{Diff}\,(M)^{G}\,.$\\
	On the other hand, if $X\in \varphi(\mathcal{U})$ is such 
	that $\varphi^{-1}(X)\in \mathcal{U}\cap\text{Diff}\,(M)^{G}\,,$ 
	then for all $g\in G\,:$
	\begin{eqnarray*}
		\vartheta_{g}\circ\big(\varphi^{-1}(X)\big)=
		\big((\varphi^{-1}(X)\big)\circ\vartheta_{g}\,\,\Rightarrow
		\,\,\vartheta_{g}(\text{exp}_{x}\,(X_{x}))=
		\text{exp}_{\vartheta_{g}(x)}\,X_{\vartheta_{g}(x)}\,\,\,\forall x\in M\,.
	\end{eqnarray*}
	Using (ii)\,, we then easily get
	$$
		X_{\vartheta_{g}(x)}=(\vartheta_{g})_{*_{x}}\,X_{x}\,\,\,\forall g\in G\,,
	$$
	i.e., $X\in \mathfrak{X}(M)^{G}\,.$ Therefore 
	$\varphi^{-1}(U_{1}\times \{0\})=\mathcal{U}\cap\text{Diff}\,(M)^{G}\,.$ 
	The group $\text{Diff}\,(M)^{G}$ is thus a tame submanifold of 
	$\text{Diff}\,(M)$ near $Id_{M}$ and by translations, 
	$\text{Diff}\,(M)^{G}$ becomes a tame Lie subgroup of 
	$\text{Diff}\,(M)^{G}\,.$
\end{proof}
\begin{proposition}
\label{SDiff G}
	If $\mu$ is a $G$-invariant volume form on $M\,,$ then the 
	group $\text{SDiff}\,(M,\mu)^{G}:=\{\varphi\in \text{SDiff}\,(M,\mu)\,
	\vert\,\vartheta_{g}\circ\varphi=\varphi\circ\vartheta_{g},\,\forall g\in G\}$ 
	is a tame Lie subgroup of both $\text{Diff}\,(M)^{G}$ and 
	$\text{SDiff}\,(M,\mu)\,.$ Its Lie algebra is the space 
	$\mathfrak{X}(M,\mu)^{G}:=\mathfrak{X}(M,\mu)\cap\mathfrak{X}(M)^{G}\,.$
\end{proposition}
	In order to show this proposition, we need the following three lemmas\,.
\begin{lemma}[Helmholtz-Hodge decomposition]\label{lemme decomposition de hodge la vraie} 
	Let $(M,h)$ be a compact, connected, oriented Riemannian manifold 
	without boundary and whose volume form $\mu=d\,vol_{h}$ is the 
	volume form induced by the metric $h\,.$ Then we have the following decomposition :
	\begin{eqnarray}\label{equation decomposition de hodge la vraie}
		\mathfrak{X}(M)=\mathfrak{X}(M,\mu)\oplus\nabla \Omega^{0}(M)\,.
	\end{eqnarray}
\end{lemma}
	 A proof of Lemma \ref{lemme decomposition de hodge la vraie} 
	 is available in \cite{Arnold}, page 341 or \cite{de-Rham}\,. Note 
	 that in the decomposition \eqref{equation decomposition de hodge la vraie}, 
	 the space  $\nabla \Omega^{0}(M)$ is isomorphic to $C^{\infty}_{0}(M,\mathbb{R})$ 
	 where $C^{\infty}_{0}(M,\mathbb{R}):=\{f\in C^{\infty}(M,\mathbb{R})\,
	 \vert\,\int_{M}\,f\,\mu=0\}\,.$
\begin{lemma}\label{Hodge} 
	Let $G$ be a connected, compact Lie group which acts by isometries 
	on a Riemannian manifold $(M,h)\,.$ We assume $M$ compact, connected 
	and oriented, the orientation being given by $\mu:=\text{d\,vol}_{h}\,.$\\
	 If $X=X^{\mu}+\nabla f$ is the Helmholtz-Hodge decomposition of a 
	 vetor field $X\in\mathfrak{X}(M)$ (i.e. $X^{\mu}\in \mathfrak{X}(M,\mu)$ 
	 and $f\in C^{\infty}_{0}(M,\mathbb{R})$)\,, then we have the 
	 following equivalence\,:
	$$
		X\in\mathfrak{X}(M)^{G}\Leftrightarrow X^{\mu}\in 
		\mathfrak{X}(M,\mu)^{G}\,\,\text{and}\,\,f\in C^{\infty}_{0}(M,\mathbb{R})^{G}\,.
	$$
	In other words,
	\begin{eqnarray}
		\mathfrak{X}(M)^{G}=\mathfrak{X}(M,\mu)^{G}\oplus C^{\infty}_{0}(M,\mathbb{R})^{G}\,,
	\end{eqnarray}
	\begin{flushleft}
		$ \text{where}\,\,\, C^{\infty}_{0}(M,\mathbb{R})^{G}:
		=\{f\in C^{\infty}_{0}(M,\mathbb{R})\,\vert\,f\circ \vartheta_{g}= f,\,\forall g\in G\}
		$ (we denote by $\vartheta\,:\,G\times M\rightarrow M$ the action of $G$ on $M$).
	\end{flushleft}
\end{lemma}
\begin{proof}
	Let $X=X^{\mu}+\nabla f\in \mathfrak{X}(M)^{G}$ be 
	the Helmholtz-Hodge decomposition of a $G$-invariant vetor field. For $g\in G\,,$ we have\,:
	\begin{eqnarray}
		\text{div}\,(X)=\text{div}\,(X^{\mu})+\triangle f\,\,\Rightarrow\,\,
		\text{div}\,(X)=\triangle f\,\,\Rightarrow\,\,\text{div}\,(X)
		\circ\vartheta_{g}=\triangle f\circ\vartheta_{g}\label{eq 3}\,.
	\end{eqnarray}
	On the other hand, as $X$ and $h$ are $G$-invariant, 
	\begin{eqnarray}
		\text{div}\,(X)\circ\vartheta_{g}=\text{div}\,(X)\,\,\,\,
		\text{and}\,\,\,\,(\triangle f)\circ\vartheta_{g}=
		\triangle (f\circ \vartheta_{g})\,.\label{je suis a lausanne!!}
	\end{eqnarray}
	From \eqref{eq 3} together with \eqref{je suis a lausanne!!}, 
	we get
	\begin{eqnarray}
		\text{div}\,(X)=\triangle (f\circ \vartheta_{g})\,.\label{eq 2}
	\end{eqnarray}
	We deduce from \eqref{eq 3} and \eqref{eq 2} that $f$ and 
	$f\circ\vartheta_{g}$ satisfy the same elliptic equation on a compact 
	connected manifold, and it is well known (see for example \cite{Jost}), 
	that the kernel of the Laplacian $\triangle$ on the space 
	$C^{\infty}(M,\mathbb{R})$ is reduced to the space of constant functions. 
	Hence $f\circ \vartheta_{g}=f+c(g)$ where $c(g)\in\mathbb{R}\,,$ and 
	as $\int_{M}\,f\,\mu=0\,,$ we must have $c(g)=0$ for all $g\in G\,,$ 
	i.e. $f\in C^{\infty}_{0}(M,\mathbb{R})^{G}\,.$ It follows that 
	$X^{\mu}=X-\nabla f\in \mathfrak{X}(M,\mu)^{G}$ since $X$ and 
	$\nabla f$ are $G$-invariant.\\
	The other implication being trivial, the lemma follows.
\end{proof}		
	Let us introduce some terminology before the second lemma. 
	Let $(\mathcal{U},\varphi)$ be the ``standard'' chart of 
	$\text{Diff}\,(M)$ near the identity element $Id_{M}$ such as	
	in the proof of Proposition \ref{diffG}, constructed from a $G$-invariant 	
	metric $h$ (note that we can take $h$ such that $\mu=d\,\text{vol}_{h}$)\,.	
	For $X\in \varphi(\mathcal{U}),$ define $P(X)\in C^{\infty}(M,\mathbb{R})$ by :		
	$$
		\big(\varphi^{-1}(X)\big)^{*}\mu=P(X)\cdot\mu\,.
	$$
	Without loss of generality, we may assume the volume form $\mu$ to be 
	normalized and take $\mathcal{U}$ such that $\int_{M}\,P(X)\,\mu=1$ 
	for all $X\in \mathcal{U}\,.$ According to the Helmholtz-Hodge 
	decomposition, we have the following direct sum 
	$$
		\mathfrak{X}(M)=\mathfrak{X}(M,\mu)\oplus C^{\infty}_{0}(M,\mathbb{R})
	$$
	which allows us to define a map	
	$$
		Q\,:\,
		\left\lbrace
		\begin{matrix}
		\varphi(\mathcal{U})\subseteq \mathfrak{X}(M)=
		\mathfrak{X}(M,\mu)\oplus C^{\infty}_{0}(M,\mathbb{R})\rightarrow 
		\mathfrak{X}(M,\mu)\oplus C^{\infty}_{0}(M,\mathbb{R})\,;\\
		(X,f)\mapsto(X,\,P(X+\nabla f)-1)\,.
		\end{matrix}
		\right.
	$$
	It is shown in \cite{Hamilton}, Theorem 2.5.3, that $Q$ is invertible 
	in a neighborhood of $0$ in $\mathfrak{X}(M)\,.$ The following 
	lemma shows also that $Q$ is compatible with the symmetries of $M\,.$
\begin{lemma}\label{preservation symetrie lemme}
		For all sufficiently small neighborhoods $K$ of $0$ in $\mathfrak{X}(M),$ we have
		\begin{eqnarray}\label{preservation symetrie}
			Q\Big(K\cap \mathfrak{X}(M)^{G}\Big)=Q(K)\cap \mathfrak{X}(M)^{G}\,.
		\end{eqnarray}
\end{lemma}
\begin{proof} 
	From the inverse function Theorem of Nash-Moser, there exists 
	$W\subseteq \mathfrak{X}(M)\,,$ a neighborhood of $0$ in
	$\mathfrak{X}(M)\,,$ $V_{1}$ a neighborhood of $0$ in $\mathfrak{X}(M,\mu)$ 
	and $V_{2}$ a neighborhood of $0$ in $C^{\infty}_{0}(M,\mathbb{R})$ such that
	$$
		Q\big\vert_{V_{1}\times V_{2}}\,:\,V_{1}\times V_{2}\rightarrow W
	$$
	is a diffeomorphism. Let us make the following two observations\,:
	\begin{description}
		\item[$\bullet$] restricting $\mathcal{U}$ if necessary, we may assume
			$\varphi(\mathcal{U})=V_{1}\times V_{2}\,,$
		\item[$\bullet$] by compactness of the group $G$ and continuity 
			of the map $G\times V_{2}\rightarrow$ $C_{0}^{\infty}(M,\mathbb{R})\,,(g,f)
			\mapsto f\circ \vartheta_{g}\,,$ we can find $\widetilde{V}_{2}
			\subseteq V_{2}$ a neighborhood of $0$ in $C^{\infty}_{0}(M,\mathbb{R})$ 
			such that if
			$f\in \widetilde{V}_{2}\,,$ then $ f\circ\vartheta_{g}\in V_{2}$
			for all $g\in G\,.$
	 \end{description}
	Let us show that the map $Q$ restricted to $(V_{1}\times\widetilde{V}_{2})
	\cap \mathfrak{X}(M)^{G}$ is a diffeomorphism from 
	$(V_{1}\times\widetilde{V}_{2})\cap \mathfrak{X}(M)^{G}$ onto 
	$Q(V_{1}\times\widetilde{V}_{2})\cap \mathfrak{X}(M)^{G}\,.$ 
	For that purpose, it is sufficient to show that
	\begin{eqnarray}\label{relationeq}
		Q\big((V_{1}\times\widetilde{V}_{2})\cap \mathfrak{X}(M)^{G}\big)=
		Q(V_{1}\times\widetilde{V}_{2})\cap \mathfrak{X}(M)^{G}\,.
	\end{eqnarray}
	According to Lemma (\ref{Hodge}), and since $h$ is $G$-invariant, 
	the inclusion from the left-handside to the right-handside of 
	\eqref{relationeq} is clear. \\
	Let us show the inverse inclusion. For $(X,\,P\big(X+\nabla f)-1\big)\in 
	Q(V_{1}\times\widetilde{V}_{2})\cap \mathfrak{X}(M)^{G}\,,$ 
	we have according to Lemma (\ref{Hodge}),
	$$
		X\in \mathfrak{X}(M,\mu)^{G}\,\,\,\,\text{and}\,\,\,\,
		P(X+\nabla f)-1\in C^{\infty}_{0}(M,\mathbb{R})^{G}\,.
	$$
	Thus, for $g\in G\,:$
	\begin{eqnarray*}
		&&\Big(P(X+\nabla\,f)-1\Big)\circ \vartheta_{g}=P(X+\nabla\,f)-1\\
		&\Rightarrow& P(X+\nabla\,f)\circ \vartheta_{g}-1=P(X+\nabla\,f)-1\\
		&\Rightarrow& P\Big((\vartheta_{g})_{*}(X+\nabla\,f)\Big)
			=P(X+\nabla\,f)\,\,\,\,\,\,\,\,\,\,(P\,\,\text{is}\,\,G\text{-invariant})\\
		&\Rightarrow& P\Big((X+\nabla\,(f\circ\vartheta_{g}))\Big)=P(X+\nabla\,f)\\
		&\Rightarrow& Q(\underbrace{X}_{\in V_{1}},\,
			\underbrace{f\circ \vartheta_{g}}_{\in V_{2}})=Q(X,\,f)\\
		&\Rightarrow& f\circ \vartheta_{g}= f\,\,\,\,\,\,\,\,\,\,(Q\,\,
			\text{is a diffeomorphism on}\,\,V_{1}\times V_{2})\,.
	\end{eqnarray*}
	Hence, $(X,\,f)\in \big(V_{1}\times\widetilde{V}_{2}\big)\cap 
	\mathfrak{X}(M)^{G}$ which implies \eqref{relationeq}. It follows 
	that \eqref{preservation symetrie} holds for all sufficiently small 
	neighborhoods $K$ of 0 in $\mathfrak{X}(M)$.
\end{proof}
\begin{proof}[Proof of Proposition \ref{SDiff G}]
	Let us recall how to construct a chart centered at $Id_{M}$ of 
	the group  $\text{SDiff}\,(M,\mu)$ using the map $Q\,.$ 
	According to the proof of Theorem 2.5.3. in \cite{Hamilton} and 
	restricting the domain $\mathcal{U}$ of the chart $(\mathcal{U},\varphi)$ 
	if necessary, we can find $K_{1}\subseteq \mathfrak{X}(M,\mu)$ and 
	$K_{2}\in C^{\infty}_{0}(M,\mathbb{R})\,,$ two neighborhoods of 0 in 
	$\mathfrak{X}(M,\mu)$ and $C^{\infty}_{0}(M,\mathbb{R})$ respectively, 
	such that $Q\,:\,\varphi(\mathcal{U})\rightarrow K_{1}\times K_{2}$ 
	becomes a diffeomorphism. Then, denoting $\mathcal{U}^{S}:=\mathcal{U}\,\cap\, 
	\text{SDiff}\,(M,\mu)\,,$ one can check that 
	$\Big(\mathcal{U}^{S},\,\big(Q\big\vert_{\varphi(\mathcal{U}^{S})}\big)
	\circ \big(\varphi\big\vert_{\mathcal{U}^{S}}\big)\Big)$ is a chart of 
	$\text{SDiff}\,(M,\mu)\,,$ i.e., 
	$\Big(\big(Q\big\vert_{\varphi(\mathcal{U}^{S})}\big)\circ 
	\big(\varphi\big\vert_{\mathcal{U}^{S}}\big)\Big)^{-1}(K_{1}
	\times\{0\})=\mathcal{U}^{S}\,.$
	On the other hand, choosing $\mathcal{U}$ sufficiently small, 
	we know from Lemma \ref{preservation symetrie} that we may also 
	assume $Q(\varphi(\mathcal{U})\cap\mathfrak{X}(M)^{G})=
	Q(\varphi(\mathcal{U}))\cap\mathfrak{X}(M)^{G}\,.$ We then get 
	the following commutative diagram :
	\begin{eqnarray}\label{diagram}
		\xymatrix@!C{
			\mathcal{U}^{S} \ar@/^2.5pc/[rr]^{\displaystyle(Q\circ \varphi)\,
				\big \vert\,_{\mathcal{U}^{S}}}\ar[r]_{\displaystyle\cong}^
				{\displaystyle\varphi\big\vert_{\mathcal{U}^{S}}}\ar@{^{(}->}[d] 
				& \varphi(\mathcal{U^{S}}) \ar@{^{(}->}[d] \ar[r]_{\displaystyle\cong}
				^{\displaystyle Q\big\vert_{\varphi(\mathcal{U^{S}}) }}& 
				K_{1}\times\{0\} \ar@{^{(}->}[d]\\
			 \mathcal{U} \ar[r]_{\displaystyle\cong}^{\displaystyle \varphi}&
				\varphi(\mathcal{U})\ar[r]_{\displaystyle\cong}^
				{\displaystyle Q} &K_{1}\times K_{2}\\
			\ar@/_2.5pc/[rr]_{\displaystyle(Q\circ \varphi)\,\big \vert\,_
				{\mathcal{U}^{G}}}\mathcal{U}^{\,G}\ar@{^{(}->}[u] 
				\ar[r]_{\displaystyle\cong}^{\displaystyle\varphi\big\vert_
				{\mathcal{U}^{G}}}&\varphi(\mathcal{U})^{G}\ar@{^{(}->}[u]
				\ar[r]_{\displaystyle\cong}^{\displaystyle 
				Q\big\vert_{\varphi(\mathcal{U})^{\,G}}}
				\ar[r]_{\displaystyle\cong}&K_{1}^{\,G}\times K_{2}^{\,G}
				\ar@{^{(}->}[u]
			}
	\end{eqnarray}
	$\text{}$\\\\
	notations being obvious, for example, $\mathcal{U}^{\,G}:=\mathcal{U}\,
	\cap\,\text{Diff}\,(M)^{\,G}\,.$ Clearly,
	$$
	\Big(\mathcal{U}^{S,G},\,\big(Q\big\vert_{\varphi(\mathcal{U}^{S,G})}\big)
	\circ \big(\varphi\big\vert_{\mathcal{U}^{S,G}}\big)=
	(Q\circ\varphi)\,\big\vert\,_{\mathcal{U}^{S,G}}\Big)
	$$
	is a chart of $\text{SDiff}\,(M,\mu)^{G}$ (where $\mathcal{U}^{S,G}\,
	:=\mathcal{U}^{S}\,\cap\,\mathcal{U}^{G}$) and therefore 
	$\text{SDiff}\,(M,\mu)^{G}$ is a submanifold of $\text{Diff}\,(M)^{G}$ 
	in a neighborhood of the identity. By translations, 
	$\text{SDiff}\,(M,\mu)^{G}$ becomes a tame Lie subgroup of $\text{Diff}\,(M)^{G}\,.$\\
	The fact that $\text{SDiff}\,(M,\mu)^{G}$ is also a Lie subgroup 
	of $\text{SDiff}\,(M,\mu)$ can be proved similarly using the same techniques 
	appearing above and in Proposition (\ref{diffG}).
\end{proof}
\section{Some integration formulas for a principal bundle}\label{chapitre 2 section 2}
	Let $G\hookrightarrow P\overset{\displaystyle\pi}{\rightarrow} B$ 
	be a principal bundle and $h^{P}$ a $G$-invariant metric on $P$ 
	(we assume that $G$ and $P$ are compact and connected). In this section, 
	we shall use the following terminology\,:
	\begin{description}
		\item[$\bullet$] $\vartheta\,:\,P\times G\rightarrow P$ is 
			the right action of the structure group $G$ on the total space $P\,,$
		\item[$\bullet$] $\mathcal{O}_{x}\subseteq P$ is the orbit 
			through the point $x\in P$ for the action $\vartheta\,,$ 
		\item[$\bullet$] given $g\in G$ and $x\in P\,,$ we write 
			$\vartheta_{g}\,:\,P\rightarrow P\,,\,x\mapsto 
			\vartheta(x,g)$ and $\vartheta_{x}\,:\,G\overset{\cong}{\rightarrow} 
			\mathcal{O}_{x}\subseteq P\,,\,g\mapsto \vartheta(x,g)$ 
			for the associated maps (note that $\vartheta_{x}$ is a 
			diffeomorphism from $G$ onto $\mathcal{O}_{x}\,,$ thus, one 
			can consider the map $\vartheta_{x}^{-1}\,:\,\mathcal{O}_{x}\rightarrow G$),
		\item[$\bullet$] if $X_{x}\in T_{x}P$ for a given point $x\in P\,,$ 
			we denote by $X^{v}$ the orthogonal projection of $X_{x}$ on 
			$T_{x}\mathcal{O}_{x}$ and $X^{h}$ the component of $X_{x}$ 
			perpendicular to $T_{x}\mathcal{O}_{x}\,,$
		\item[$\bullet$] the Lie algebra of the group $G$ is denoted by 
			$\mathfrak{g}\,,$
	\end{description}
	The metric $h^{p}$ being $G$-invariant, we naturally get an induced 
	connection form $\theta\in \Omega^{1}(P,\,\mathfrak{g})$ which is 
	defined, for $x\in P$ and $X_{x}\in T_{x}P\,,$ by\,:
	\begin{eqnarray}
		\theta_{x}(X_{x}):=(\vartheta_{x}^{-1})_{*_{x}}\,X_{x}^{v}\in \mathfrak{g}\,.
	\end{eqnarray}
	In particular, one can check that
	\begin{eqnarray}\label{invariance connection}
		(\vartheta_{g})^{*}\,\theta=Ad(g^{-1})\,\theta\,,
	\end{eqnarray}
	for all $g\in G\,.$
	Recall also that for any vector field $Z\in\mathfrak{X}(B)\,,$ there 
	exists a unique horizontal lift $Z^{*}\in \mathfrak{X}(P)^{G}$ satisfying
	$\pi_{*_{x}}\,Z^{*}_{x}=Z_{\pi(x)}$ for all $x\in P$ 
	(see \cite{Kobayashi-Nomizu})\,.\\ The following easy lemma describes more 
	precisely the metric $h^{P}\,.$
\begin{lemma}\label{metrique sur P}
	There exists a metric $h^{B}$ on $B$ and an Euclidean structure 
	$h^{\mathfrak{g}}$ on the trivial bundle $P\times\mathfrak{g}$ such that :
	\begin{description}
		\item[$(i)$] $h^{P}_{x}(X_{x},Y_{x})=(\pi^{*}\,h^{B})_{x}(X_{x},Y_{x})+
			h^{\mathfrak{g}}_{x}(\theta_{x}(X_{x}),\theta_{x}(Y_{x}))$ 
			for all $x\in P$ and for all $X_{x},Y_{x}\in T_{x}P\,,$
		\item[$(ii)$] $\pi\,:\,(P,h^{P})\rightarrow(B,h^{B})$ 
			is a Riemannian submersion,
		\item[$(iii)$] $h^{\mathfrak{g}}_{\vartheta_{g}(x)}(\xi,\zeta)=
			h^{\mathfrak{g}}_{x}(\text{Ad}\,(g)\,\xi,\text{Ad}\,(g)\,\zeta)$ 
			for all $g\in G,\,x\in P$ and $\xi,\zeta\in \mathfrak{g}\,.$
	 \end{description}
\end{lemma}
\begin{remarque}
	The point (i) of Lemma \ref{metrique sur P} gives a decomposition of 
	the metric $h^{P}$ with respect to the horizontal and vertical 
	tangent vectors of $P\,.$
\end{remarque}
\begin{remarque}\label{remarque metrique}
	For $x\in P\,,$ $\text{ker}\,(\pi_{*_{x}})=T_{x}
	\mathcal{O}_{x}$ and the map 
	$\pi_{*_{x}}\,\big\vert_{ (T_{x}\mathcal{O}_{x})^{\perp}}
	\,:\,(T_{x}\mathcal{O}_{x})^{\perp}\rightarrow T_{\pi(x)}B$ is 
	a linear isomorphism (see \cite{Kobayashi-Nomizu})\,. In particular, 
	there exists a canonical isomorphism between the space of $G$-invariant 
	horizontal vector fields on $P$ and the space of vector fields on  $B\,.$
\end{remarque}
	Now, let us assume that $P$ and $B$ are oriented and let us denote 
	by $\mu^{P}$ and $\mu^{B}$ the natural volume forms induced 
	respectively on $P$ and $B$ by the metrics $h^{P}$ and $h^{B}\,.$ 
	As for the metric $h^{P}\,,$ we want to give a precise description 
	of the volume form $\mu^{P}\,.$
\begin{lemma}\label{metrique volume}
	Let $(E,\,h)$ be an Euclidean oriented vector space of finite 
	dimension. We assume that $E=E_{1}\oplus E_{2}$ and also that 
	$h=p_{1}^{*}h^{E_{1}}+p_{2}^{*}h^{E_{}}$ where $h^{E_{i}}$ is a 
	metric on $E_{i}$ and $p_{i}\,:\,E_{1}\oplus E_{2}\rightarrow E_{i}$ 
	the canonical projection.\\
	If $E_{1}$ is endowed with a given orientation, then 
	$$
		\mu^{E}=p_{1}^{*}\,\mu^{E_{1}}\wedge p_{2}^{*}\,\mu^{E_{2}}\,,
	$$
	where $\mu^{E}\,,\mu^{E_{i}}$ are the volume forms associated to the 
	metrics $h,\,h^{E_{i}}$ respectively (we adopt the following convention: 
	a basis $\{f_{1},...,f_{m}\}$ of $E_{2}$ is positive if and only 
	if the family $\{e_{1},...,e_{n},f_{1},...,f_{m}\}$
	is a positive basis of $E$ whenever $\{e_{1},...,e_{n}\}$ is a 
	positive basis of $E_{1}$).
\end{lemma}
\begin{proof}
	Let $\{e_{1},...,e_{n}\}$ be a positive basis for $E_{1}$ and 
	$\{f_{1},...,f_{m}\}$ a positive basis for $E_{2}\,,$ the corresponding 
	dual basis being respectively $\{e_{1}^{*},...,e_{n}^{*}\}$ 
	and $\{f_{1}^{*},...,f_{m}^{*}\}\,.$ We introduce also 
	$h_{ij}^{E_{1}}:=h^{E_{1}}(e_{i},e_{j})$ for $i,j\in \{1,...,n\}$ 
	and $h_{ij}^{E_{2}}:=h^{E_{2}}(f_{i},f_{j})$ for $i,j\in \{1,...,m\}\,.$\\
	From the definition of the volume form induced by a metric, we have
	$$
		\mu^{E}=\big(\text{det}\,(h_{ij}^{E_{1}})\big)^{\frac{1}{2}}\,
		\big(\text{det}\,(h_{ij}^{E_{2}})\big)^{\frac{1}{2}}\,e_{1}^{*}
		\wedge\cdots\wedge e_{n}^{*}\wedge f_{1}^{*}\wedge\cdots\wedge f_{m}^{*}\,.
	$$
	On the other hand,
	\begin{eqnarray*}
		\mu^{E_{1}}=\text{det}\,(h_{ij}^{E_{1}})^{\frac{1}{2}}\,e_{1}^{*}\big
		\vert_{E_{1}}\wedge\cdots\wedge e_{n}^{*}\big\vert_{E_{1}}\,\,
		\Rightarrow\,\, p_{1}^{*}\,\mu^{E_{1}}=
		\text{det}\,(h_{ij}^{E_{1}})^{\frac{1}{2}}\,e_{1}^{*}\wedge
		\cdots\wedge e_{n}^{*}\\
	\end{eqnarray*}
	and similarly, $p_{2}^{*}\,\mu^{E_{2}}=
	\text{det}\,(h_{ij}^{E_{2}})^{\frac{1}{2}}\,f_{1}^{*}\wedge\cdots\wedge f_{n}^{*}\,.$ 
	Hence,
	\begin{eqnarray*}
		p_{1}^{*}\,\mu^{E_{1}}\wedge p_{2}^{*}\,\mu^{E_{2}}
		=\big(\text{det}\,(h_{ij}^{E_{1}})\big)^{\frac{1}{2}}\,
		\big(\text{det}\,(h_{ij}^{E_{2}})\big)^{\frac{1}{2}}\,e_{1}^{*}
		\wedge\cdots\wedge e_{n}^{*}\wedge f_{1}^{*}\wedge\cdots\wedge 
		f_{m}^{*}=\mu^{E}\,.
	\end{eqnarray*}
	This proves the lemma.
\end{proof}
	Let us apply Lemma \ref{metrique volume} to $\mu^{P}\,.$ 
	For $x\in P\,,$ we write\,:
	\begin{description}
		\item[$\bullet$] $E_{1}:=(T_{x}\mathcal{O}_{x})^{\perp}\,;\,\,\,\,
			E_{2}:=T_{x}\mathcal{O}_{x}\,,$
		\item[$\bullet$] $h^{E_{1}}_{x}(\xi_{1},\,\xi_{2}):=
			(\pi^{*}h^{M})_{x}(\xi_{1},\,\xi_{2})$ 
			for $\xi_{1},\xi_{2}\in E_{1}\,,$
		\item[$\bullet$]  $h^{E_{2}}_{x}(\xi_{1},\,\xi_{2}):=
			h^{\mathfrak{g}}_{x}(\theta_{x}(\xi_{1}),\,
			\theta_{x}(\xi_{2}))$ for $\xi_{1},\xi_{2}\in E_{2}\,.$
	\end{description}
	For $i\in\{1,2\}\,,$ $h^{E_{i}}$ is a metric on $E_{i}$ and we 
	have $h_{x}^{P}=p_{1}^{*}h^{E_{1}}+p_{2}^{*}h^{E_{}}$ 
	where $p_{i}\,:\,E_{1}\oplus E_{2}\rightarrow E_{i}$ is the canonical 
	projection. Since we assume the manifold $B$ oriented, the space 
	$E_{1}$ is also oriented by the isomorphism 
	$\pi_{*_{x}}\big\vert_{E_{1}}\rightarrow T_{\pi(x)}B\,.$ 
	We fix on $E_{2}$ the orientation given by Lemma \ref{metrique volume}. We then have :
	\begin{eqnarray}\label{eq forme volume en x}
		\mu^{P}_{x}=p_{1}^{*}\,\mu^{E_{1}}\wedge p_{2}^{*}\,\mu^{E_{2}}\,.
	\end{eqnarray}
\begin{remarque}\label{orientation}
	Orientations on $P$ and $B$ induce an orientation on $G$ in the 
	following way : for $x\in P\,,$ the spaces  $T_{x}\mathcal{O}_{x}$ 
	and $\mathfrak{g}$ are isomorphic via the map 
	$\theta_{x}\big\vert_{T_{x}\mathcal{O}_{x}}\,:\,
	T_{x}\mathcal{O}_{x}\rightarrow\mathfrak{g}\,.$ But, the 
	space $T_{x}\mathcal{O}_{x}$ being oriented (see above), the Lie 
	algebra $\mathfrak{g}$ is also oriented and induces an orientation 
	on $G\,.$ This orientation doesn't depend on the point $x\in P\,.$ 
	In fact, if $\mu^{\mathfrak{g}}_{x}$ is the volume form on 
	$\mathfrak{g}$ induced by the metric $h^{\mathfrak{g}}_{x}\,,$ 
	it is obvious that $\mu^{\mathfrak{g}}_{x}$ depends continuously 
	of the point $x\in P\,,$ and the orientation induced by 
	$\mu^{\mathfrak{g}}_{x}$ cannot be reversed.
\end{remarque}
\begin{lemma}\label{lemme p1} 
	With the above notations, we have
	\begin{eqnarray}
		p_{1}^{*}\,\mu^{E_{1}}=(\pi^{*}\mu^{B})_{x}\,.
	\end{eqnarray}
\end{lemma}
\begin{proof} 
	Let $(U,\varphi)$ be a positive chart of $B$ containing 
	$\pi(x)$ with local coordinates $\{x_{1},...,x_{n}\}\,.$ 
	This gives a positive basis for $E_{1}\,:$
	$$
		\bigg\{\Big(\pi_{*_{x}}\big\vert_{E_{1}}\Big)^{-1}\,
		\dfrac{\partial}{\partial x_{i}}
		\bigg\vert_{\pi(x)}\,,\,\,\,i=1,...,n\bigg\}\,.
	$$
	For $i\in\{1,...,n\}$\,, define
	$$
		e_{i}:=\Big(\pi_{*_{x}}\big\vert_{E_{1}}\Big)^{-1}\,
		\dfrac{\partial}{\partial x_{i}}\bigg\vert_{\pi(x)}\,.
	$$
	We have
	$$
		\mu^{E_{1}}=\text{det}\,(h^{E_{1}}_{ij})^{\frac{1}{2}}\,
		e_{1}^{*}\big\vert_{E_{1}}\wedge\cdots\wedge e_{n}^{*}\big\vert_{E_{1}}
	$$
	with
	\begin{eqnarray}
		h^{E_{1}}_{ij}=h^{E_{1}}(e_{i},\,e_{j})=
			h^{B}_{\pi(x)}(\pi_{*_{x}}\,e_{i},\,\pi_{*_{x}}\,e_{j})
			=h^{B}_{\pi(x)}\bigg(\dfrac{\partial}{\partial x_{i}}
			\bigg\vert_{\pi(x)},\,\dfrac{\partial}{\partial x_{j}}
			\bigg\vert_{\pi(x)}\bigg)=(h^{B}_{ij}\circ \pi)(x)\,.
	\end{eqnarray}
	Hence, $p_{1}^{*}\,\mu^{E_{1}}=(\text{det}\,(h^{B}_{ij})^{\frac{1}{2}}
	\circ\pi)(x)\,e_{1}^{*}\wedge\cdots\wedge e_{n}^{*}\,.$\\
	On the other hand,
	\begin{eqnarray}
		(\pi^{*}\mu^{B})_{x}(e_{1},...,e_{n})=
			\mu^{B}_{\pi(x)}(\pi_{*_{x}}\,e_{1},...,\pi_{*_{x}}\,e_{n})
			=\mu^{B}_{\pi(x)}\bigg(\dfrac{\partial}{\partial x_{1}}
			\bigg\vert_{\pi(x)},...,\dfrac{\partial}{\partial x_{n}}\bigg\vert_{\pi(x)}\bigg)
			=(\text{det}\,(h^{B}_{ij})^{\frac{1}{2}}\circ\pi)(x)
	\end{eqnarray}
	which implies that
	$$
		(\pi^{*}\mu^{B})_{x}=(\text{det}\,(h^{B}_{ij})^{\frac{1}{2}}
		\circ\pi)(x)\,e_{1}^{*}\wedge\cdots\wedge e_{n}^{*}=p_{1}^{*}\,\mu^B\,.
	$$
	Thus, $p_{1}^{*}\,\mu^{E_{1}}=(\pi^{*}\mu^{B})_{x}\,.$
\end{proof}	
\begin{lemma}\label{lemme p2}
	With the notations introduced before Lemma \ref{lemme p1}, we have 
	\begin{eqnarray}
		p_{2}^{*}\,\mu^{E_{2}}=\theta^{*}_{x}\,\mu^{\mathfrak{g}}_{x}\,,
	\end{eqnarray}
	where $\mu^{\mathfrak{g}}_{x}$ is the volume form on 
	$\mathfrak{g}$ induced by the metric $h^{\mathfrak{g}}_{x}$ 
	(see Remark \ref{orientation}) and where $\theta^{*}_{x}\,
	\mu^{\mathfrak{g}}_{x}$ is the pullback of 
	$\mu^{\mathfrak{g}}_{x}$ by the linear map 
	$\theta_{x}\,:\,T_{x}P\rightarrow \mathfrak{g}\,.$
\end{lemma}
\begin{proof} 
	Let $\{\xi_{1},...,\xi_{m}\}$ be a positive basis for 
	$\mathfrak{g}$ (see Remark \ref{orientation} for the question of 
	the orientation of $\mathfrak{g}$). The family 
	$\{(\vartheta_{x})_{*_{e}}\,\xi_{1},...,(\vartheta_{x})_{*_{e}}\,\xi_{m}\}$ 
	is a positive basis for $E_{2}$ and we have the formula\,:
	\begin{eqnarray}\label{j'ai pas d'inspiration}
		\mu^{E_{2}}=\text{det}\,(h_{ij}^{E_{2}})^{\frac{1}{2}}\,
		\Big((\vartheta_{x})_{*_{e}}\,\xi_{1}\Big)^{\flat}
		\wedge\cdots\wedge\Big((\vartheta_{x})_{*_{e}}\,\xi_{1}\Big)^{\flat}\,
	\end{eqnarray}
	where $``\,\overset{\flat}{\text{}}\,"\,:\,E_{2}
	\rightarrow E_{2}^{*}$ denotes the ``dualisation'' operator 
	with respect to the metric $h_{x}^{E_{2}}\,.$
	But,
	\begin{eqnarray}
		h^{E_{2}}_{ij}=h^{E_{2}}\Big((\vartheta_{x})_{*_{e}}\,\xi_{i},\,
			(\vartheta_{x})_{*_{e}}\,\xi_{j}\Big)
		=h^{\mathfrak{g}}_{x}\Big(\theta_{x}\big((\vartheta_{x})_{*_{e}}\,
			\xi_{i}\big),\,\theta_{x}\big((\vartheta_{x})_{*_{e}}\,
			\xi_{j}\big)\Big)
		=h_{x}^{\mathfrak{g}}(\xi_{i},\xi_{j})=
			(h^{\mathfrak{g}}_{x})_{ij}\label{sgdsgf}
	\end{eqnarray}
	and one can check, for $u\in E_{2}\,,$ that
	\begin{eqnarray}\label{vkjhlksjh}
		\Big((\vartheta_{x})_{*_{e}}\,\xi_{j}\Big)^{\flat}u=
		\xi_{j}^{\flat}\big(\theta_{x}(u)\big)
		\Rightarrow \Big((\vartheta_{x})_{*_{e}}\,\xi_{j}\Big)^{\flat}
		=\Big((\vartheta^{-1}_{x})_{*_{e}}\Big)^{*}\,\xi_{j}^{\flat}\,.
	\end{eqnarray}
	From \eqref{sgdsgf} and \eqref{vkjhlksjh} applied to 
	\eqref{j'ai pas d'inspiration}, it follows that :
	\begin{eqnarray*}
		\mu^{E_{2}}&=&\text{det}\,\Big((h_{x}^{\mathfrak{g}})_{ij}
			\Big)^{\frac{1}{2}}\,\bigg(\Big((\vartheta^{-1}_{x})_{*_{e}}
			\Big)^{*}\,\xi_{1}^{\flat}\bigg)\wedge\cdots\wedge\bigg
			(\Big((\vartheta^{-1}_{x})_{*_{x}}\Big)^{*}\,\xi_{m}^{\flat}\bigg)\\
		&=&\text{det}\,\Big((h_{x}^{\mathfrak{g}})_{ij}\Big)^
			{\frac{1}{2}}\,\Big((\vartheta^{-1}_{x})_{*_{e}}\Big)^{*}
			\,(\xi_{1}^{\flat}\wedge\cdots\wedge\xi_{m}^{\flat})\\
		&=&\Big((\vartheta^{-1}_{x})_{*_{e}}\Big)^{*}\,\mu_{x}^{\mathfrak{g}}\,.
	\end{eqnarray*}
	Finally,
	\begin{eqnarray*}
		p_{2}^{*}\,\mu^{E_{2}}=p_{2}^{*}\,\Big((\vartheta^{-1}_{x})_{*_{e}}
		\Big)^{*}\,\mu_{x}^{\mathfrak{g}}=\Big((\vartheta^{-1}_{x})_{*_{e}}
		\circ p_{2}\Big)^{*}\,\mu_{x}^{\mathfrak{g}}=\theta_{x}^{*}\,\mu_{x}^{\mathfrak{g}}
	\end{eqnarray*}
	which is the desired formula.
\end{proof}
	From Lemma \ref{lemme p1} and Lemma \ref{lemme p2}, it follows, 
	using formula  \eqref{eq forme volume en x}, that
	\begin{eqnarray}
		\mu^{P}_{x}=(\pi^{*}\mu^{B})_{x}\wedge\theta_{x}^{*}\,\mu_{x}^{\mathfrak{g}}\,.
	\end{eqnarray}
	Let us consider the unique normalized volume form 
	$\nu^{G}$ on $G$ (note that $\nu^{G}$ is bi-invariant since $G$ is compact and connected). 
	For $x\in P\,,$ let $\widetilde{V}(x)$ be the unique real number satisfying 
	$\widetilde{V}(x)\cdot\nu^{G}_{e}=\mu_{x}^{\mathfrak{g}}\,.$ 
	The $G$-invariance of $\mu^{P}$ implies the existence of a 
	function $V\in \text{C}^{\infty}(B,\mathbb{R}_{+}^{*})$ 
	such that $\widetilde{V}=V\circ \pi\,.$ To summarize,
\begin{proposition}
	\label{proposition formule integration}
	There exists a function $V\in \text{C}^{\infty}(B,\mathbb{R}_{+}^{*})$ 
	such that 
	\begin{eqnarray}\label{equation formule forme volume}
		\mu^{P}=\pi^{*}(V\mu^{B})\wedge\theta^{*}\,\nu_{e}^{G}\,,
	\end{eqnarray}
	where $\theta^{*}\,\nu_{e}^{G}\in \Omega^{m}(P)$ 
	($m=\text{dim}\,(G)$) is defined by
	$$
		(\theta^{*}\,\nu_{e}^{G})_{x}(X_{1},...,X_{m})=
		\nu_{e}^{G}\Big(\theta_{x}(X_{1}),...,\theta_{x}(X_{m})\Big)\,,
	$$
	for any $x\in P$ and $X_{1},...,X_{m}\in T_{x}P\,.$
\end{proposition}
	In order to give a geometrical interpretation to the function 
	$V\,,$ let us make the following remark.
\begin{remarque}
	For $x\in P\,,$ the orbit $\mathcal{O}_{x}$ of $P$ through the 
	point $x$ is canonically oriented via the orbit map 
	$\vartheta_{x}\,:\,G\overset{\cong}{\rightarrow} \mathcal{O}_{x}\,.$ 
	This orientation on $\mathcal{O}_{x}$ doesn't depend of the orbit map 
	which is used because, for $g\in G\,,$ the connectedness of $G$ 
	implies that the map 
	$\vartheta_{\vartheta_{g}(x)}\,:\,G\overset{\cong}{\rightarrow} 
	\mathcal{O}_{x}$ induces the same orientation. Thus, we can consider 
	without ambiguities the volume form 
	$\mu^{\mathcal{O}_{x}}$ of $\mathcal{O}_{x}$ induced by the 
	restriction of the metric $h^{P}$ on $\mathcal{O}_{x}\,.$
\end{remarque}
\begin{lemma}
	For $x\in P\,,$ we have the formula :
	\begin{eqnarray}
		\mu^{\mathcal{O}_{x}}=(V\circ\pi)(x)\cdot(\vartheta^{-1}_{x})^{*}\,\nu^{G}\,.
	\end{eqnarray}
	In particular, $V\big(\pi(x)\big)=\text{Vol}\,(\mathcal{O}_{x})\,.$
\end{lemma}
\begin{proof}
	Let $f\in C^{\infty}(\mathcal{O}_{x},\mathbb{R})$ be the unique map satisfying 
	\begin{eqnarray}\label{equation par hypothese}
		(V\circ\pi)(x)\cdot(\vartheta^{-1}_{x})^{*}\,\nu^{G}=f\cdot\mu^{\mathcal{O}_{x}}\,.
	\end{eqnarray}
	Let $g\in G$ be arbitrary. The forms $\nu^{G}$ and 
	$\mu^{\mathcal{O}_{x}}$ being $G$-invariant, we have :
	\begin{eqnarray}
		(V\circ\pi)(x)\cdot(\vartheta^{-1}_{x})^{*}\,\nu^{G}=
			f\cdot\mu^{\mathcal{O}_{x}}
		&\Rightarrow& \vartheta_{g}^{*}\Big((V\circ\pi)(x)\cdot
			(\vartheta^{-1}_{x})^{*}\,\nu^{G}\Big)=
			\vartheta_{g}^{*}\Big(f\cdot\mu^{\mathcal{O}_{x}}\Big)\nonumber\\
		&\Rightarrow&(V\circ\pi)(x)\cdot(\underbrace{\vartheta^{-1}_{x}
			\circ\vartheta_{g}}_{\displaystyle=L_{g}\circ\vartheta_{x}^{-1}})^{*}\,
			\nu^{G}=f\circ\vartheta_{g}\cdot\mu^{\mathcal{O}_{x}}\nonumber\\
		&\Rightarrow&(V\circ\pi)(x)\cdot(\vartheta^{-1}_{x})^{*}\,\nu^{G}=
		f\circ\vartheta_{g}\cdot\mu^{\mathcal{O}_{x}}\label{equation apres symetrie}\,.
	\end{eqnarray}
	From \eqref{equation par hypothese} and \eqref{equation apres symetrie}, 
	it follows that $f\circ\vartheta_{g}=f$ for all $g\in G\,.$ This 
	implies that $f$ is constant on $\mathcal{O}_{x}\,.$\\
	Let us show that $f(x)=1\,.$ Let $\{u_{1},...,u_{m}\}$ be an orthonormal 
	positive basis for $T_{x}\mathcal{O}_{x}$ (we assume that the 
	dimension of $G$ is equal to $m$). Observe that the map 
	$\theta_{x}\big\vert_{T_{x}\mathcal{O}_{x}}=(\vartheta_{x}^{-1})_{*_{x}}\,:\,
	\big(T_{x}\mathcal{O}_{x},\,h^{P}\big\vert_{T_{x}\mathcal{O}_{x}}\big)
	\rightarrow(\mathfrak{g}\,,\,h^{\mathfrak{g}}_{x})$ is :
	\begin{description}
		\item[$\bullet$] an isometry according to (i) in 
			Lemma \ref{metrique sur P}\,,
		\item[$\bullet$] an isomorphism which preserves the orientation 
			according to Remark \ref{orientation}\,.
	\end{description}
	It follows that $\big\{(\vartheta_{x}^{-1})_{*_{x}}\,u_{1},...,
	(\vartheta_{x}^{-1})_{*_{x}}\,u_{m}\big\}$ is an orthonormal 
	basis of $\mathfrak{g}$ and
	\begin{eqnarray*}
		&&\Big(f(x)\cdot\mu^{\mathcal{O}_{x}}\Big)_{x}(u_{1},...,u_{m})
			=\Big((V\circ\pi)(x)\cdot(\vartheta^{-1}_{x})^{*}\,
			\nu^{G}\Big)_{x}(u_{1},...,u_{m})\\
		&\Rightarrow& f(x)=(V\circ\pi)(x)\cdot\nu^{G}_{e}\,
			\Big((\vartheta_{x}^{-1})_{*_{x}}\,u_{1},...,(\vartheta_{x}^{-1})_{*_{x}}\,u_{m}\Big)\\
		&\Rightarrow&f(x)=\mu^{\mathfrak{g}}_{x}\,
			\Big((\vartheta_{x}^{-1})_{*_{x}}\,u_{1},...,(\vartheta_{x}^{-1})_{*_{x}}\,u_{m}\Big)\\
		&\Rightarrow& f(x)=1\,.
	\end{eqnarray*}
	The lemma follows.
\end{proof}
	Before the end of this section, let us give an integration formula.
\begin{proposition}\label{proposition vraiment formule integration}
	For $f\in C^{\infty}(B,\mathbb{R})\,,$ we have the following formula :
	\begin{eqnarray}
		\int_{P}\,(f\circ\pi)\cdot\mu^{P}=\int_{B}\, f\cdot V\mu^{B}\,.
	\end{eqnarray}
\end{proposition}
	Proposition \ref{proposition vraiment formule integration} can be 
	shown using two lemmas.
\begin{lemma}\label{lemme pas assez de place}
	Let $E_{1},E_{2}$ be two vector spaces of respective 
	dimension $n$ and $m\,,$ $\mu\in (\Lambda^{n}E_{1}^{*}){\setminus}\{0\}$ 
	and $p_{i}\,:\,E:=E_{1}\times E_{2}\rightarrow E_{i}$ 
	the canonical projection associated ($i=1,2$)\,. 
	For $\alpha\in \Lambda^{m}E^{*}\,,$ we have :
	\begin{eqnarray}
		p_{1}^{*}\,\mu\wedge\alpha=p_{1}^{*}\,\mu\wedge\widetilde{\alpha}\,,
	\end{eqnarray}
	where $\widetilde{\alpha}\in \Lambda^{m}E^{*}$ is defined, 
	for $(u_{1},v_{1}),...,(u_{m},v_{m})\in E,$ by :
	\begin{eqnarray}\label{equation le secret de l'integrale}
	\widetilde{\alpha}\Big((u_{1},v_{1}),...,(u_{m},v_{m})\Big):=
	\alpha\Big((0,v_{1}),...,(0,v_{m})\Big)\,.
	\end{eqnarray}
\end{lemma}
\begin{proof}
	Let $\{x_{1},...,x_{n}\}$ be a basis of $E_{1}\,,$ $\{y_{1},...,y_{m}\}$ 
	a basis of $E_{2}$ and let $\{z_{1},...,z_{n+m}\}$ 
	denote the basis of $E$ canonically associated, i.e., $\{z_{1},...,z_{n+m}\}:=
	\{(x_{1},0),...,$ $(x_{n},0),(0,y_{1}),...,(0,y_{m})\}\,.$ We can write 
	\begin{eqnarray}
		 \mu=\kappa\cdot x_{1}^{*}\wedge\cdots\wedge 
			x_{n}^{*}\,\,\,\,\,(\kappa\in \mathbb{R}^{*})\,\,\,\,\,\textup{and}\,\,\,\,\,
			\displaystyle\alpha=\sum_{1\leq i _{1}<\cdots
			<i_{m}\leq n+m}\,\alpha_{i_{1}...i_{m}}
			\cdot z_{i_{1}}^{*}\wedge\cdots\wedge z_{i_{m}}^{*}\,.
	\end{eqnarray}
	We then have,
	\begin{eqnarray}\label{equation je sais pas quoi mettre}
		p_{1}^{*}\,\mu\wedge \alpha&=&\kappa\cdot 
			z_{1}^{*}\wedge\cdots\wedge z_{n}^{*}\wedge
			\sum_{1\leq i _{1}<\cdots<i_{m}\leq n+m}\,
			\alpha_{i_{1}...i_{m}}\,z_{i_{1}}^{*}\wedge\cdots\wedge z_{i_{m}}^{*}\nonumber\\
		&=&\kappa\,\sum_{1\leq i _{1}<\cdots<i_{m}\leq n+m}\,
			\alpha_{i_{1}...i_{m}}\cdot z_{1}^{*}\wedge\cdots
			\wedge z_{n}^{*}\wedge z_{i_{1}}^{*}\wedge\cdots\wedge z_{i_{m}}^{*}\nonumber\\
		&=&\kappa\,\alpha_{n+1...n+m}\cdot z_{1}^{*}\wedge\cdots
			\wedge z_{n}^{*}\wedge z_{n+1}^{*}\wedge\cdots\wedge 
			z_{n+m}^{*}\nonumber\,\,\,\,\,\,\,\,\,\,\,\,\,\,\,\,\,\,\text{}\\
		&=&\kappa\,\alpha\Big((0,y_{1}),...,(0,y_{m})\Big)
			\cdot z_{1}^{*}\wedge\cdots\wedge z_{n+m}^{*}
			\,.\,\,\,\,\,\,\,\,\,\,\,\,\,\,\,\,\,\,\,\,\,\,\,\,\,\,\,\,\,\,\text{}
	\end{eqnarray}
	On the other hand, if
	$$
		\widetilde{\alpha}=\sum_{1\leq i _{1}<\cdots<i_{m}\leq n+m}\,
		\widetilde{\alpha}_{i_{1}...i_{m}}\cdot z_{i_{1}}^{*}\wedge\cdots\wedge z_{i_{m}}^{*}\,,
	$$
	then, according to \eqref{equation le secret de l'integrale}\,,
	\begin{eqnarray}
		\widetilde{\alpha}_{i_{1}...i_{m}}=
		\widetilde{\alpha}\Big(z_{i_{1}},...,z_{i_{m}}\Big)
		=
		\left\lbrace  \label{equation technique}
		\begin{array}{cc}
			\alpha\Big((0,y_{1}),...,(0,y_{m})\Big)\,,&\,\,\, 
			\text{for}\,\,\, (i_{1},...,i_{m})=(1,...,m)\,,\\
			0&\text{otherwise}\,.
		\end{array}
		\right.
	\end{eqnarray}
	The equality between \eqref{equation je sais pas quoi mettre} 
	and $p_{1}^{*}\,\mu\wedge\widetilde{\alpha}$ now follows from 
	\eqref{equation technique}.
\end{proof}
	For the second lemma, we fix a local trivialization 
	$(U,\varphi)$ of $B\,:$
	$$
		\begin{diagram}
			\node{\pi^{-1}(U)} \arrow[2]{e,t}{\displaystyle\Psi} 
			\arrow{se,b}{\displaystyle\pi} \node[2]{U\times G} 
			\arrow{sw,b}{\displaystyle pr_{1}}\\
			\node[2]{U}
		\end{diagram}\,\,\,,
	$$
	(the map $\Psi$ being $G$-equivariant).
\begin{lemma}\label{lemme la cle}
	We have 
	\begin{eqnarray}
		(\Psi^{-1})^{*}\mu^{P}=(V\circ pr_{1})\cdot \big(pr_{1}^{*}\,
		\mu^{B}\big)\wedge\big(pr_{2}^{*}\,\nu^{G}\big)\,.
	\end{eqnarray}
\end{lemma}
\begin{proof}
	From \eqref{equation formule forme volume}, 
	\begin{eqnarray}
		(\Psi^{-1})^{*}\mu^{P}&=&(\Psi^{-1})^{*}\,\Big((V\circ\pi)
			\cdot\pi^{*}\mu^{B}\wedge\theta^{*}\,\nu_{e}^{G}\Big)\nonumber\\
		&=&(V\circ\pi\circ\Psi^{-1})\cdot\bigg(\Big((\Psi^{-1})^{*}\pi^{*}
			\mu^{B}\Big)\wedge\Big((\Psi^{-1})^{*}\theta^{*}\,
			\nu_{e}^{G}\Big)\bigg)\nonumber\\
		&=&(V\circ pr_{1})\cdot\bigg(\Big(pr_{1}^{*}\,\mu^{B}\Big)
			\wedge\Big((\Psi^{-1})^{*}\theta^{*}\,\nu_{e}^{G}\Big)\bigg)
			\label{equation formule penible}\,.
	\end{eqnarray}
	For $(x,g)\in U\times G\,,$ $u_{1},...,u_{m}\in T_{x}B$ 
	and $\xi_{1},...,\xi_{m}\in T_{g}G$ (we assume the dimension 
	of $G$ equal to $m$), we have :
	\begin{eqnarray}
		&& \Big((\Psi^{-1})^{*}\theta^{*}\,\nu_{e}^{G}\Big)_{(x,g)}
			\Big((u_{1},\xi_{1}),...,(u_{m},\xi_{m})\Big)\nonumber\\
		&=& \Big(\theta^{*}\,\nu_{e}^{G}\Big)_{(\Psi^{-1})(x,g)}
			\Big(\Psi^{-1}_{*_{(x,g)}}(u_{1},\xi_{1}),...,
			\Psi^{-1}_{*_{(x,g)}}(u_{m},\xi_{m})\Big)\label{equation qwertzzjvvh}\\
		&=&\nu_{e}^{G}\bigg(\theta_{(\Psi^{-1})(x,g)}
			\Big(\Psi^{-1}_{*_{(x,g)}}(u_{1},\xi_{1})\Big),...,
			\theta_{(\Psi^{-1})(x,g)}\Big(\Psi^{-1}_{*_{(x,g)}}(u_{m},
			\xi_{m})\Big)\bigg)\nonumber\,.
	\end{eqnarray}
	Let $s\,:\,U\rightarrow P$ be the local section which characterizes 
	the trivialization $\Psi\,,$ i.e., 
	$$
	\Psi^{-1}(x,g)=\vartheta_{g}\big(s(x)\big)=\vartheta\big(s(x),g\big)\,,
	$$
	for all $(x,g)\in U\times G\,.$ For $i\in\{1,...,m\}\,,$ 
	we have 
	\begin{eqnarray*}
		\Psi^{-1}_{*_{(x,g)}}(u_{i},\xi_{i})&=&
		\begin{bmatrix}
			(\vartheta_{g})_{*}s_{*_{x}}&(\vartheta_{s(x)})_{*_{g}}
		\end{bmatrix}
		\begin{bmatrix}
		u_{i}\\
		\xi_{i}
		\end{bmatrix}\\
		&=&(\vartheta_{g})_{*}s_{*_{x}}u_{i}+(\vartheta_{s(x)})_{*_{g}}\xi_{i}\,\,,
	\end{eqnarray*}
	which yields, together with \eqref{invariance connection}\,,
	\begin{eqnarray}
		\theta_{(\Psi^{-1})(x,g)}\Big(\Psi^{-1}_{*_{(x,g)}}
			(u_{i},\xi_{i})\Big)&=&\theta_{\vartheta_{g}\big(s(x)\big)}
			\Big((\vartheta_{g})_{*}s_{*_{x}}u_{i}\Big)+
			\theta_{\vartheta_{g}\big(s(x)\big)}
			\Big((\vartheta_{s(x)})_{*_{g}}\xi_{i}\Big)\nonumber\\
		&=&Ad(g^{-1})\,\theta_{s(x)}\Big(s_{*_{x}}u_{i}\Big)
			+\theta_{\vartheta_{g}\big(s(x)\big)}
			\Big((\vartheta_{s(x)})_{*_{g}}\xi_{i}\Big)\,.\label{equation calcul qui tue}
	\end{eqnarray}
	We can notice in formula \eqref{equation calcul qui tue}\,, that
	\begin{eqnarray}
		\theta_{\vartheta_{g}\big(s(x)\big)}\Big((\vartheta_{s(x)})_{*_{g}}
			\xi_{i}\Big)&=&(\vartheta_{\vartheta_{s(x)}(g)}^{-1})_{*_{\vartheta_{g}
			(s(x))}}(\vartheta_{s(x)})_{*_{g}}\xi_{i}\nonumber\\
		&=&\Big(\underbrace{\vartheta_{\vartheta_{s(x)}(g)}^{-1}
			\circ \vartheta_{s(x)}}_{=L_{g^{-1}}}\Big)_{*_{g}}\xi_{i}
			=(L_{g^{-1}})_{*_{g}}\xi_{i}\,.\label{equation translation a gauche}
	\end{eqnarray}
	It follows, taking $u_{i}=0$ in \eqref{equation qwertzzjvvh}, that :
	\begin{eqnarray*}
		&&\nu_{e}^{G}\bigg(\theta_{(\Psi^{-1})(x,g)}
			\Big(\Psi^{-1}_{*_{(x,g)}}(0,\xi_{1})\Big),...,
			\theta_{(\Psi^{-1})(x,g)}\Big(\Psi^{-1}_{*_{(x,g)}}(0,\xi_{m})\Big)\bigg)\\
		&=& \nu_{e}^{G}\bigg((L_{g^{-1}})_{*_{g}}\xi_{1},...,(L_{g^{-1}})_{*_{g}}
			\xi_{m}\bigg)=\nu_{g}^{G}\bigg(\xi_{1},...,\xi_{m}\bigg)\\
		&=&\Big(pr_{2}^{*}\,(\nu^{G})\Big)_{(x,g)}\Big((0,\xi_{1}),...,(0,\xi_{m})\Big)\,.
	\end{eqnarray*}
	One concludes by applying Lemma \ref{lemme pas assez de place}\,.
\end{proof}
\begin{proof}[Proof of Proposition \ref{proposition vraiment formule integration}]
	Let $\big\{(U_{i},\varphi_{i})\,\big\vert\,i\in \{1,...,s\}\big\}$ 
	be an atlas of $B$ whose charts $(U_{i},\varphi_{i})$ are positive 
	and trivializing:
	$$
	\begin{diagram}
		\node{\pi^{-1}(U_{i})} \arrow[2]{e,t}{\displaystyle\Psi_{i}} 
		\arrow{se,b}{\displaystyle\pi} \node[2]{U_{i}\times G} 
		\arrow{sw,b}{\displaystyle pr_{1}^{i}}\\
		\node[2]{U_{i}}
	\end{diagram}\,\,\,.
	$$
	We also take $\Big\{\Big(\big(U_{i},\varphi_{i}\big),\alpha_{i}\Big)\,
	\Big\vert\,i\in \{1,...,s\}\Big\}$ a partition of unity of 
	$B$ subordinate to $\big\{U_{i}\,\big\vert\,i\in \{1,...,s\}\big\}\,.$\\
	We take on $U_{i}\times G$ the orientation induced by the volume 
	form $ \big((pr_{1}^{i})^{*}(V\mu^{B})\big)\wedge
	\big((pr_{2}^{i})^{*}\,\nu^{G}\big)\,.$ For this orientation, 
	$\Psi_{i}$ is a diffeomorphism which preserves orientation and 
	from Lemma \ref{lemme la cle}, we have 
	\begin{eqnarray*}
		\int_{P}\,(f\circ\pi)\cdot\mu^{P}&=&\sum_{i=1}^{s}\,
			\int_{P}\,(\alpha_{i}\circ\pi)\cdot (f\circ\pi)\cdot\mu^{P}=
			\sum_{i=1}^{s}\,\int_{\pi^{-1}(U_{i})}\,(\alpha_{i}\circ\pi)
			\cdot (f\circ\pi)\cdot\mu^{P}\,\,\,\,\,\,\,\,\,\text{}\nonumber\\
		&=&\sum_{i=1}^{s}\int_{U_{i}\times G}\big((\alpha_{i}{\cdot} f)
			\circ\pi\circ\Psi_{i}^{-1}\big){\cdot}(\Psi_{i}^{-1})^{*}\mu^{P}=
			\sum_{i=1}^{s}\int_{U_{i}\times G}\,
			\big((\alpha_{i}{\cdot} f){\circ} pr_{1}^{i}\big)
			{\cdot}(\Psi_{i}^{-1})^{*}\mu^{P}\\
		&=&\sum_{i=1}^{s}\,\int_{U_{i}\times G}\,\big((\alpha_{i}\cdot f)
			\circ pr_{1}^{i}\big)\cdot \Big((pr_{1}^{i})^{*}\,
			(V\mu^{B})\Big)\wedge\Big((pr_{2}^{i})^{*}\,\nu^{G}\Big)\\
		&=&\sum_{i=1}^{s}\,\text{Volume}\,(G)\,\int_{U_{i}}\,
			\alpha_{i}\cdot f\cdot V\mu^{B}=\int_{B}\, f\cdot V\mu^{B}\,.
	\end{eqnarray*} 
	This proves the proposition.
\end{proof}
\section{The Euler equation of $\text{SAut}\,(P,\mu^{P})$}\label{section equation euler}
	For a Fr\'echet Lie algebra $(\mathfrak{g},\,[\,,\,])$ endowed 
	with a continuous, symmetric, weakly non-degenerate and 
	positive-definite bilinear form $<\,,\,>\,,$ we define the regular 
	dual $\mathfrak{g}_{reg}^{*}\subseteq \mathfrak{g}^{*}$ of 
	$\mathfrak{g}$ as the range of the injective and continuous 
	operator $\mathfrak{g}\rightarrow \mathfrak{g}^{*},\,\xi\rightarrow {<}
	\xi,\,.\,{>}\,.$ For $\xi\in \mathfrak{g}\,,$ we also define the 
	operator $ad^{*}(\xi)\,:\,\mathfrak{g}_{reg}^{*}\rightarrow 
	\mathfrak{g}^{*}$ via the formula :
	\begin{eqnarray}
		(ad^{*}(\xi)\,\alpha,\xi'):=-(\alpha,ad(\xi)\,\xi')\,,
	\end{eqnarray}
	where $\alpha\in \mathfrak{g}_{reg}^{*}$ and $\xi'\in\mathfrak{g}\,.$
	Observe that the range of $ad^{*}(\xi)$ is not necessarily 
	included in $\mathfrak{g}_{reg}^{*}$ 
	(it is the case, for example if $ad(\xi)$ possesses a 
	transpose with respect to the metric $<\,,\,>$).
\begin{definition}
	If $ad^{*}(\xi)$ takes values in $\mathfrak{g}_{reg}^{*}$ for 
	all $\xi\in \mathfrak{g}\,,$ we define the Euler equation 
	associated to the Lie algebra $\mathfrak{g}$ with respect 
	to the metric $<\,,\,>$ as :
	\begin{eqnarray}\label{definition equation poisson}
		\dfrac{d}{dt}\eta =ad^{*}(\eta^{\sharp})\,\eta\,,
	\end{eqnarray}
	where $\eta$ is a smooth curve in $\mathfrak{g}_{reg}^{*}$ and 
	where $``\,\,\overset{\sharp}{\text{}}\,\,"\,:\,
	\mathfrak{g}^{*}_{reg}\rightarrow\mathfrak{g}$ denotes the 
	canonical operator induced by the metric $<\,,\,>\,.$
\end{definition}
\begin{remarque}\label{remarque sur le lien EUler geodesiques}
	If $g$ is a geodesic in a finite dimensional Lie group $G$ with respect to a right-invariant 
	metric $<\,,\,>\,,$ then the curve $\eta:=[(R_{g^{-1}})_{*g}\dot{g}]^{\flat}\,,$ 
	where $R_{g_{-1}}\,:\,G\rightarrow G\,,h\mapsto hg^{-1}\,,$ is a curve in 
	$\mathfrak{g}^{*}$ satisfying the 
	Euler equation \eqref{definition equation poisson}\,. Conversely, 
	if $\eta$ is a curve in $\mathfrak{g}^{*}$ satisfying \eqref{definition equation poisson},
	then one may recover a geodesic in $G$ via the 
	``reconstruction procedure", i.e., by solving a specific first order differential equation 
	(see \cite{Marsden-Ratiu} for more details). 
	The geodesic equation on a Lie group with
	respect to a right-invariant metric and the Euler equation 
	\eqref{definition equation poisson} are thus equivalent.
\end{remarque}
	We want next to determine the Euler equation of the Lie algebra 
	of the group 
	\begin{eqnarray}
	\textup{SAut}(P,\mu^{P}):=
	\{\varphi\in\textup{Diff}(P)\,|\,\varphi^{*}\mu^{P}=
	\mu^{P}\,\,\text{and}\,\,\varphi\,\,\text{is}\,\,G\text{-equivariant}\}\,,
	\end{eqnarray}
	with 
	respect to a natural 
	$L^{2}$-metric (see \eqref{equation def de la metrique})\,. Note that 
	$\textup{SAut}(P,\mu^{P})=\textup{SDiff}(P,\mu^{P})^{G}$ 
	and thus it is a tame Lie group by Proposition \ref{SDiff G}\,, and its 
	Lie algebra is the space $\mathfrak{X}(P,\mu^{P})^{G}$ endowed with the 
	opposite of the usual vector field bracket. Note also that 
	$\textup{SAut}(P,\mu^{P})=\textup{Aut}(P)\cap \textup{SDiff}(P,\mu^{P})$ where 
	\begin{eqnarray}
		\textup{Aut}(P):=\{\varphi\in\textup{Diff}(P)\,\vert\,
		\varphi\,\,\text{is}\,\,G\text{-equivariant}\}
	\end{eqnarray}
	is the group of smooth automorphisms of $P\,.$

\subsection{The identification of $\mathfrak{X}(P,\mu^{P})^{G}$ 
	and $\mathfrak{X}(B,V \mu^{B})\oplus C^{\infty}(P,\mathfrak{g})^{G}$}
	Set 
	$$
		C^{\infty}(P,\mathfrak{g})^{G}:=\{f\in C^{\infty}
		(P,\mathfrak{g})\,\big\vert\,f\circ 
		\vartheta_{g}=Ad(g^{-1})\,f\,\,,\forall g\in G\}
	$$
	and define $\Phi\,:\,\mathfrak{X}(P)^{G}\rightarrow
	\mathfrak{X}(B)\oplus C^{\infty}(P,\mathfrak{g})^{G}$ as :
	\begin{eqnarray}\label{definition phi}
		\Phi(X):=\Big(\pi_{*}X^{h},\,\theta(X^{v})\Big)\,,
	\end{eqnarray}
	where $X\in \mathfrak{X}(P)^{G}$ and where $\pi_{*}X^{h}\in 
	\mathfrak{X}(B)$ denotes the vector field defined for 
	$x=\pi(y)\in B\,,$ by $(\pi_{*}X^{h})_{x}:=\pi_{*_{y}}X^{h}_{y}\,.$ 
	One can check using Remark \ref{remarque metrique} and 
	\eqref{invariance connection} that $\Phi$ is well defined and 
	invertible, the inverse being given by $\big(\Phi^{-1}(X,f)\big)_{x}
	=X^{*}_{x}+(\vartheta_{x})_{*_{e}}f(x)\,,$ where 
	$X\in \mathfrak{X}(B)\,,$ $f\in C^{\infty}(P,\mathfrak{g})^{G}$ 
	and $x\in P\,.$ The space $\mathfrak{X}(P)^{G}$ being a Lie algebra, 
	$\mathfrak{X}(B)\oplus C^{\infty}(P,\mathfrak{g})^{G}$ 
	naturally inherits a Lie algebra structure. More precisely, 
\begin{proposition}\label{proposition identification crochet}
	The Lie bracket of the Lie algebra 
	$\mathfrak{X}(B)\oplus C^{\infty}(P,\mathfrak{g})^{G}$ is given by :
		\begin{eqnarray}
			\big[(Z,f),(Z',f')\big]=-\bigg([Z,Z'],\,[f,f']+Z^{*}(f')-(Z')^{*}(f)
			+\Omega\Big(Z^{*},\,(Z')^{*}\Big)\bigg)
			\,\,\,\,\,\,\,\,\,\,\,\,\text{}\label{equation le crochet de lie}
		\end{eqnarray}
	where $\Omega\in \Omega^{2}(P,\mathfrak{g})$ is the curvature of 
	the connection $\theta\,,$ i.e., $\Omega_{x}(X,Y)=
	\theta_{x}([X^{h},Y^{h}])$ for $x\in P$ and $X,Y\in T_{x}P\,.$
\end{proposition}
\begin{remarque}
	The minus sign appearing in front of the term 
	\eqref{equation le crochet de lie} comes from the fact that we 
	consider on $\mathfrak{X}(P)^{G}$ the Lie bracket induced by the 
	Lie group structure of $\text{Aut}(P)\,.$
\end{remarque}
	Let us give some lemmas to prove this result.
\begin{lemma}\label{lemme champ vertical}
	Let $X,Y\in \mathfrak{X}(P)^{G}$ be $G$-invariant vector fields with $Y$ vertical. We have 
	\begin{eqnarray}
		[X,Y]_{x}=(\vartheta_{x})_{*_{e}}X_{x}\Big(\theta(Y)\Big)\,,
	\end{eqnarray}
	where $x\in P\,.$
\end{lemma}
\begin{proof}
	We have 
	\begin{eqnarray}\label{equation sous forme de derivee}
		[X,Y]_{x}=\dfrac{d}{dt}\bigg\vert_{0}\,(\varphi_{-t}^{X})
		_{*_{\varphi_{t}^{X}(x)}}Y_{\varphi_{t}^{X}(x)}\,.
	\end{eqnarray}
	Moreover, $Y$ being vertical,
	\begin{eqnarray}\label{equation formule flot}
		Y_{\varphi_{t}^{X}(x)}=\dfrac{d}{ds}\bigg\vert_{0}\,
		\vartheta\Big(\varphi_{t}^{X}(x),\,\text{exp}\,
		\big(s\,\theta_{\varphi_{t}^{X}(x)}(Y)\big)\Big)\,.
	\end{eqnarray}
	Using the $G$-invariance of $X$ together with 
	\eqref{equation formule flot} in 
	\eqref{equation sous forme de derivee}, we get :
	\begin{eqnarray*}
		[X,Y]_{x}&=&\dfrac{d}{dt}\bigg\vert_{0}\,\dfrac{d}{ds}
			\bigg\vert_{0}\,(\varphi_{-t}^{X})
			\bigg(\,\vartheta\Big(\varphi_{t}^{X}(x),\,\text{exp}\,
			\big(s\,\theta_{\varphi_{t}^{X}(x)}(Y)\big)\Big)\bigg)\\
		&=&\dfrac{d}{dt}\bigg\vert_{0}\,\dfrac{d}{ds}\bigg\vert_{0}\,
			\vartheta\Big(x,\,\text{exp}\,\big(s\,
			\theta_{\varphi_{t}^{X}(x)}(Y)\big)\Big)=
			\dfrac{d}{dt}\bigg\vert_{0}\,(\vartheta_{x})_{*_{e}}
			\theta_{\varphi_{t}^{X}(x)}(Y)\\
		&=&(\vartheta_{x})_{*_{e}}\,\dfrac{d}{dt}\bigg\vert_{0}\,
			\theta_{\varphi_{t}^{X}(x)}(Y)=
			(\vartheta_{x})_{*_{e}}X_{x}\Big(\theta(Y)\Big)\,,
	\end{eqnarray*}
	which proves the lemma.
\end{proof}
	The following lemma is proved in \cite{Kobayashi-Nomizu}.
\begin{lemma}\label{lemme reference kn}
	Let $X,Y\in \mathfrak{X}(P)^{G}$ be $G$-invariant vector 
	fields and $x\in P\,.$ We have :
	\begin{eqnarray}
		[X^{h},Y^{h}]_{x}=[(\pi_{*}X^{h}),\,(\pi_{*}Y^{h})]_{x}^{*}+
		(\vartheta_{x})_{*_{e}}\Omega_{x}(X^{h},Y^{h})\,.
	\end{eqnarray}
\end{lemma}
\begin{proof}[Proof of Proposition \ref{proposition identification crochet}]
	Let $X,Y\in \mathfrak{X}(P)^{G}$ be two $G$-invariant vector 
	fields and $x\in P\,.$ From Lemma \ref{lemme champ vertical} 
	and Lemma \ref{lemme reference kn}\,, we get :
	\begin{eqnarray*}
		[X,Y]_{x}&=&[X^{h},Y^{h}]_{x}+[X^{h},Y^{v}]_{x}+[X^{v},
			Y^{h}]_{x}+[X^{v},Y^{v}]_{x}\\
		&=&\Big[(\pi_{*}X^{h}),\,(\pi_{*}Y^{h})\Big]_{x}^{*}+
			(\vartheta_{x})_{*_{e}}\Omega_{x}(X^{h},Y^{h})+
			(\vartheta_{x})_{*_{e}}X^{h}_{x}\Big(\theta(Y^{v})\Big)\\
		&&-(\vartheta_{x})_{*_{e}}Y^{h}_{x}
			\Big(\theta(X^{v})\Big)+\underbrace{\Big[\,(\vartheta_{x})_{*_{e}}
			\theta(X^{v})\,,\,(\vartheta_{x})_{*_{e}}
			\theta(Y^{v})\,\Big]_{x}}_{\displaystyle=
			(\vartheta_{x})_{*_{e}}[\theta(X^{v})\,,\,\theta(Y^{v})]}\\
		&=&\bigg(\Phi^{-1}\bigg(\Big[(\pi_{*}X^{h}),\,
			(\pi_{*}Y^{h})\Big]\,,\,\Big[\theta(X^{v})\,,\,\theta(Y^{v})\Big]+
			X^{h}_{x}\Big(\theta(Y^{v})\Big)-
			Y^{h}_{x}\Big(\theta(X^{v})\Big)+
			\Omega_{x}(X^{h},Y^{h})\bigg)\bigg)_{x}\,,
	\end{eqnarray*}
	which is the desired formula.
\end{proof}
	For $G$-invariant vector fields on $P$ with zero divergence 
	with respect to the volume form $\mu^{P}\,,$ we have the following proposition.
\begin{proposition}\label{proposition isomorphime avec s}
	The map $\Phi$ induces a $\mathbb{R}$-linear isomorphism :
	\begin{eqnarray}
		\mathfrak{X}(P,\mu^{P})^{G}\cong\mathfrak{X}(B,V \mu^{B})
		\oplus C^{\infty}(P,\mathfrak{g})^{G}\,,
	\end{eqnarray}
	i.e., if $X\in \mathfrak{X}(P)^{G}\,,$ then 
	$X\in \mathfrak{X}(P,\mu^{P})^{G}$ if and only if 
	$\Phi(X)\in \mathfrak{X}(B,V \mu^{B})\oplus C^{\infty}(P,\mathfrak{g})^{G}\,.$
\end{proposition}
	To show Proposition \ref{proposition isomorphime avec s}, we need the 
	following Lemma :
\begin{lemma}\label{lemme algebre de lie propriete}
	For $X\in \mathfrak{X}(P)^{G}\,,$ we have :
	\begin{description}
		\item[$(i)$]  $X(V\circ\pi)=\Big((\pi_{*}X)(V)\Big)\circ\pi\,,$
		\item[$(ii)$] $\mathcal{L}_{X}(\pi^{*}\mu^{B})=
			\pi^{*}\Big(\mathcal{L}_{\pi_{*}X}(\mu^{B})\Big)=
			\big(\text{div}_{\mu^{B}}\,(\pi_{*}X)\circ\pi\big)\cdot \pi^{*}\mu^{B}\,,$
		\item[$(iii)$] $(\pi^{*}\mu^{B})\wedge\mathcal{L}_{X}
			(\theta^{*}\nu_{e}^{G})=0\,.$
	 \end{description}
\end{lemma}
\begin{proof}
	The point (i) is obvious. Let us show (ii)\,. 
	Using the relation 
	$\pi\circ \varphi_{t}^{X}=\varphi_{t}^{\pi_{*}X}\circ \pi\,,$ we see that
	\begin{eqnarray*}
		\mathcal{L}_{X}(\pi^{*}\mu^{B})&=&\dfrac{d}{dt}\bigg\vert_{0}\,
			(\varphi_{t}^{X})^{*}\pi^{*}\mu^{B}=
			\dfrac{d}{dt}\bigg\vert_{0}\,(\pi\circ\varphi_{t}^{X})^{*}\mu^{B}=
			\dfrac{d}{dt}\bigg\vert_{0}\,(\varphi_{t}^{\pi_{*}X}\circ \pi)^{*}\mu^{B}\\
		&=&\dfrac{d}{dt}\bigg\vert_{0}\,\pi^{*}(\varphi_{t}^{\pi_{*}X})^{*}\mu^{B}
			=\pi^{*}\,\dfrac{d}{dt}\bigg\vert_{0}\,(\varphi_{t}^{\pi_{*}X})^{*}\mu^{B}=
			\pi^{*}\Big(\mathcal{L}_{\pi_{*}X}(\mu^{B})\Big)\,.\\
	\end{eqnarray*}
	For $(iii)\,,$ let us take $x\in P$ and $X,Y\in 
	\mathfrak{X}(P)^{G}$ with $Y$ vertical. We have,
	\begin{eqnarray}
		(\mathcal{L}_{X}\theta)_{x}(Y_{x})=\dfrac{d}{dt}\bigg\vert_{0}\,
			\Big((\varphi_{t}^{X})^{*}\theta\Big)_{x}(Y_{x})=
			\dfrac{d}{dt}\bigg\vert_{0}\,\theta_{\varphi_{t}^{X}(x)}
			\Big((\varphi_{t}^{X})_{*_{x}}Y_{x}\Big)=
			\dfrac{d}{dt}\bigg\vert_{0}\,
			(\vartheta_{\varphi_{t}^{X}(x)}^{-1})
			_{*_{\varphi_{t}^{X}(x)}}\,(\varphi_{t}^{X})_{*_{x}}Y_{x}\,.
			\label{equation derivee de lie}
	\end{eqnarray}
	If in \eqref{equation derivee de lie}, we look at $Y_{x}$ as an 
	element of $T_{x}\mathcal{O}_{x}$ and $\varphi_{t}^{X}$ as a 
	diffeomorphim between $\mathcal{O}_{x}$ and $\mathcal{O}_{\varphi_{t}^{X}(x)}\,,$ then,
	\begin{eqnarray*}
		(\mathcal{L}_{X}\theta)_{x}(Y_{x})=
		\dfrac{d}{dt}\bigg\vert_{0}\,(\underbrace{\vartheta_{\varphi_{t}^{X}(x)}^{-1}
		\circ \varphi_{t}^{X}}_{{\displaystyle=
		\vartheta_{x}^{-1}}})_{*_{x}}Y_{x}=
		\dfrac{d}{dt}\bigg\vert_{0}\,(\vartheta_{x}^{-1})_{*_{x}}Y_{x}=0\,.
	\end{eqnarray*}
	Hence, the form $\mathcal{L}_{X}(\theta^{*}\nu_{e}^{G})$ 
	only depends on horizontal vector fields of $P\,.$ 
	Now $(iii)$ follows from Lemma \ref{lemme pas assez de place}\,.
\end{proof}
\begin{proof}[Proof of Proposition \ref{proposition isomorphime avec s}]
	Let $X\in \mathfrak{X}(P)^{G}$ be a $G$-invariant vector field. 
	From \eqref{equation formule forme volume} together with Lemma 
	\ref{lemme algebre de lie propriete}, we have:
	\begin{eqnarray*}
		\mathcal{L}_{X}\mu^{P}&=&\mathcal{L}_{X}\Big((V\circ\pi)\cdot\pi^{*}
			\mu^{B}\wedge\theta^{*}\,\nu_{e}^{G}\Big)\\
		&=&X(V\circ\pi)\cdot\pi^{*}\mu^{B}\wedge\theta^{*}\,
			\nu_{e}^{G}+(V\circ\pi)\cdot\mathcal{L}_{X}
			(\pi^{*}\mu^{B})\wedge\theta^{*}\,\nu_{e}^{G}+
			(V\circ\pi)\cdot\pi^{*}\mu^{B}\wedge\mathcal{L}_{X}
			(\theta^{*}\,\nu_{e}^{G})\\
		&=&\Big((\pi_{*}X)(V)\Big)\circ\pi\cdot\dfrac{1}{V\circ\pi}
			\cdot\mu^{P}+\big(\text{div}_{\mu^{B}}\,(\pi_{*}X)
			\circ\pi\big)\cdot\mu^{P}\\
		&=&\bigg(\Big((\pi_{*}X)(V)\Big)\cdot\dfrac{1}{V}+
			\text{div}_{\mu^{B}}\,(\pi_{*}X)\bigg)\circ\pi\cdot\mu^{P}\\
		&=&\Big(\text{div}_{V \mu^{B}}\,(\pi_{*}X)\Big)\circ\pi\cdot\mu^{P}\,.
	\end{eqnarray*}
	Hence,
	$$
		\text{div}_{\mu^{P}}\,(X)=\Big(\text{div}_{V \mu^{B}}\,(\pi_{*}X)\Big)\circ\pi\,,
	$$
	which proves the proposition.
\end{proof}
	Finally, $\mathfrak{X}(B,V \mu^{B})$ and $C^{\infty}(P,\mathfrak{g})^{G}$ 
	being closed subspaces of the Fr\'echet spaces 
	$\mathfrak{X}(B)$ and $C^{\infty}(P,\mathfrak{g})$ respectively, 
	we naturally get a structure of Fr\'echet space on 
	$\mathfrak{X}(B,V \mu^{B})\oplus\, C^{\infty}(P,\mathfrak{g})^{G}\,.$ 
	We denote by $\widetilde{\Phi}\,:\,\mathfrak{X}(P,\mu^{P})^{G}\rightarrow 
	\mathfrak{X}(B,V \mu^{B})\oplus C^{\infty}(P,\mathfrak{g})^{G}$ 
	the restriction of $\Phi$ to $\mathfrak{X}(P,\mu^{P})\,.$
\begin{lemma}\label{lemme la continuite de l isomorphsime}
	The map $\widetilde{\Phi}$ is a continuous $\mathbb{R}-$linear 
	isomorphism between Fr\'echet spaces.
\end{lemma}
\begin{proof}
	From Proposition \ref{proposition isomorphime avec s}, 
	we know that $\widetilde{\Phi}$ is a bijection. Let us show that 
	$\widetilde{\Phi}$ is continuous. If $\alpha$ is a smooth curve 
	of $\mathfrak{X}(P,\mu^{P})^{G}\,,$ then, according to the 
	characterization of smooth curves in a space of sections 
	(see \cite{Kriegl-Michor}, Lemma 30.8.),  
	and also from the definition of $\Phi$ (see \eqref{definition phi}), 
	it comes out that $\Phi\circ \alpha$ is a smooth curve of 
	$\mathfrak{X}(B,V \mu^{B})\oplus C^{\infty}(P,\mathfrak{g})^{G}\,.$ 
	This implies that $\widetilde{\Phi}$ is smooth, in particular, 
	$\widetilde{\Phi}$ is continuous. In a similar way, one can prove 
	that $\widetilde{\Phi}^{-1}$ is also continuous.
\end{proof}
\begin{remarque}
	It follows from Proposition \ref{proposition identification crochet} 
	and Lemma \ref{lemme la continuite de l isomorphsime} 
	that $\mathfrak{X}(P,\mu^{P})^{G}$ and 
	$\mathfrak{X}(B,V \mu^{B})\oplus C^{\infty}(P,\mathfrak{g})^{G}$ 
	are isomorphic in the category of Fr\'echet Lie algebras.
\end{remarque}
\subsection{The regular dual of $\mathfrak{X}(P,\mu^{P})^{G}$}
	Let $<\,,\,>$ be the scalar product on 
	$\mathfrak{X}(P,\mu^{P})^{G}$ defined as
	\begin{eqnarray}\label{equation def de la metrique}
		{<}X,Y{>}:=\int_{P}\,h^{P}_{x}(X_{x},Y_{x})\cdot \mu^{P}\,,
	\end{eqnarray}
	where $X,Y\in \mathfrak{X}(P,\mu^{P})^{G}\,.$ 
	This scalar product induces a metric on 
	$\mathfrak{X}(B,V \mu^{B})\oplus C^{\infty}(P,\mathfrak{g})^{G}$ 
	via the map $\Phi$ (see Proposition \ref{proposition isomorphime avec s}) :
	\begin{eqnarray*}
		{<}(X,f),(X',f'){>}:=\int_{P}\,h^{P}_{x}(\Phi^{-1}(X,f)_{x},\,
		\Phi^{-1}(X',f')_{x})\cdot\mu^{P}\,,
	\end{eqnarray*}
	where $(X,f),(X',f')\in \mathfrak{X}(B,V \mu^{B})
	\oplus C^{\infty}(P,\mathfrak{g})^{G}\,.$ A more explicit 
	description of this metric can be given using Lemma 
	\ref{metrique sur P} and Proposition \ref{proposition vraiment formule integration} :
	\begin{eqnarray}
		^{}{<}(X,f),(X',f'){>}&=&\int_{P}\,h^{P}_{x}\Big(X_{x}^{*}
			+(\vartheta_{x})_{*_{x}}\,f(x),\,(X')_{x}^{*}
			+(\vartheta_{x})_{*_{x}}\,f'(x)\Big)\cdot\mu^{P}\nonumber\\
		&=&\int_{P}\,(\pi^{*}h^{B})_{x}\Big(X_{x}^{*},(X')_{x}^{*}\Big)
			\cdot\mu^{P}+\int_{P}\,h_{x}^{\mathfrak{g}}(f(x),f'(x))
			\cdot\mu^{P}\nonumber\\
		&=&\int_{P}\,h^{B}_{\pi(x)}\Big(X_{\pi(x)},(X')_{\pi(x)}\Big)
			\cdot\mu^{P}+\int_{P}\,h_{x}^{\mathfrak{g}}
			(f(x),f'(x))\cdot\mu^{P}\nonumber\\
		&=&\int_{B}\,h^{B}_{x}\Big(X_{x},(X')_{x}\Big)\cdot 
			V\mu^{B}+\int_{P}\,h_{x}^{\mathfrak{g}}(f(x),f'(x))
			\cdot\mu^{P}\,.\label{equation la porte vient de s'ouvrir}
	\end{eqnarray}
	Denoting $X^{\flat}:=h^{B}(X,\,.\,)\in \Omega^{1}(B)$ the 
	differential form ``dual'' to $X$ and  $f^{\flat}:=
	h^{\mathfrak{g}}(f,\,.\,)\in C^{\infty}(P,\mathfrak{g}^{*})^{G}=
	\{f\in C^{\infty}(P,\mathfrak{g}^{*})\,\vert\,f\circ \vartheta_{g}=
	\text{Ad}^{*}(g^{-1})\,f\,,\forall g\in G\}\,,$ we can 
	rewrite \eqref{equation la porte vient de s'ouvrir} as :
	\begin{eqnarray}\label{equation l'accouplement bestiale}
		 {<}(X,f),(X',f'){>}=\int_{B}\,X^{\flat}(X')\cdot 
		 V \mu^{B}+\int_{P}\,(f^{\flat}(x),f'(x))\cdot\mu^{P}\,,
	\end{eqnarray}
	where $(\,.\,,\,.\,)$ denotes the pairing between 
	$\mathfrak{g}$ and $\mathfrak{g}^{*}\,.$\\
	Set $A\,:\,\mathfrak{X}(B,V \mu^{B})\oplus C^{\infty}(P,\mathfrak{g})^{G}
	\rightarrow\Big(\mathfrak{X}(B,V \mu^{B})\oplus C^{\infty}
	(P,\mathfrak{g})^{G}\Big)^{*}$ to be the continuous and injective 
	dualisation operator defined as $A\big((X,f)\big):= 
	{<}(X,f)\,,\,.\,{>}$ (``\,$\text{}^{*}$\," being the topological dual)\,.
\begin{definition}\label{definition la premiere}
	We define the regular dual $\Big(\mathfrak{X}(B,V \mu^{B})
	\oplus C^{\infty}(P,\mathfrak{g})^{G}\Big)_{\text{reg}}^{*}$ 
	of $\mathfrak{X}(B,V \mu^{B})\oplus C^{\infty}(P,\mathfrak{g})^{G}$ 
	as the range of the operator $A$ in the full topological dual 
	of $\mathfrak{X}(B,V \mu^{B})\oplus C^{\infty}(P,\mathfrak{g})^{G}\,.$
\end{definition}
\begin{proposition} \label{proposition dual regulier} 
	We have an isomorphism of Fr\'echet spaces 
	\begin{eqnarray}
		\Big(\mathfrak{X}(B,V \mu^{B})\oplus 
		C^{\infty}(P,\mathfrak{g})^{G}\Big)_{\text{reg}}^{*}
		\underset{\cong}{\overset{\displaystyle \Psi}{\longrightarrow}}
		\dfrac{\Omega^{1}(B)}{d\,\Omega^{0}(B)}\oplus
		\,C^{\infty}(P,\mathfrak{g}^{*})^{G}\,,
	\end{eqnarray}
	where $\Psi$ is defined for $(X,f)\in \mathfrak{X}(B,V \mu^{B})
	\oplus C^{\infty}(P,\mathfrak{g})^{G}$ by 
	\begin{eqnarray}
		\Psi\Big(A\big((X,f)\big)\Big):=\Big([X^{\flat}],f^{\flat}\Big)\,.
	\end{eqnarray}
\end{proposition}
	We will show Proposition \ref{proposition dual regulier} using two lemmas. 
	The first lemma is a slight generalization of the Helmholtz-Hodge 
	decomposition (see Lemma \ref{lemme decomposition de hodge la vraie})\,.
\begin{lemma}[Helmholtz-Hodge decomposition]\label{lemme hodge ameliore}
	Let $(M,g)$ be a compact, connected, oriented Riemannian manifold 
	without boundary, endowed with the volume form induced by 
	$g\,.$ for $f\in C^{\infty}(M,\mathbb{R}_{+}^{*})\,,$ 
	we have the following decomposition :
	\begin{eqnarray}\label{equation hodge}
		\mathfrak{X}(M)=\mathfrak{X}(M,\mu)\oplus f\nabla \Omega^{0}(M)\,.
	\end{eqnarray}
\end{lemma}
\begin{proof}
	Let $X\in \mathfrak{X}(M)$ be a vector field and assume 
	that the decomposition \eqref{equation hodge} exists. 
	Thus, we can write $X=X^{\mu}+f \nabla p$ for $X^{\mu}\in
	\mathfrak{X}(M,\mu)\,,$ $p\in \Omega^{0}(M)\,,$ and we have :
	\begin{eqnarray}\label{equation hodge encore}
		\text{div}_{\mu}(X)=\text{div}_{\mu}(f \nabla p)=(df)(\nabla p)+f \triangle p\,.
	\end{eqnarray}
	Let $\tilde{f}\,:\,[0,1]\rightarrow C^{\infty}(M,\mathbb{R}_{+}^{*})$ 
	be a continuous path such that $\tilde{f}_{0}\equiv 1$ and 
	$\tilde{f}_{1}=f\,.$ For $t\in[0,1]\,,$ we also denote 
	$I_{t}\,:\,C^{\infty}(M,\mathbb{R})\rightarrow C_{0}^{\infty}(M,\mathbb{R}):=
	\{h\in C^{\infty}(M,\mathbb{R})\,\vert\,\int_{M}\,h\cdot\mu=0\}$ 
	the operator defined for $p\in C^{\infty}(M,\mathbb{R})\,,$ by 
	$I_{t}(p):=(d\tilde{f}_{t})(\nabla p)+\tilde{f}_{t} \triangle p\,.$ 
	It comes out that $I_{t}$ is a continuous path of elliptic 
	operators (acting on a suitable Sobolev space), and for 
	$t\in [0,1]\,,$ the kernel of $I_{t}$ is 1-dimensional 
	(this comes from the fact that locally, $I_{t}$ is without 
	constant terms, and for that kind of elliptic operators, 
	the kernel reduces to constant functions, see \cite{Jost}). 
	Moreover, it is well known that on a compact orientable Riemannian 
	manifold, the operator $\Delta\,:\,C^{\infty}(M,\mathbb{R})
	\rightarrow C_{0}^{\infty}(M,\mathbb{R})$ is surjective 
	(see for example \cite{Hebey}). Thus $Ind(I_{1})=Ind(I_{0})=1-0=1.$ 
	It follows that $I_{1}$ is surjective and in particular, 
	equation \eqref{equation hodge encore} possesses a unique 
	solution defined modulo a constant. If we take a function $p$ as 
	a solution of \eqref{equation hodge encore}, it is straightforward 
	to check that $X=(X-f \nabla p)+f \nabla p$ is the desired 
	decomposition.
\end{proof}	
	The second lemma concerns the topology we put on the space 
	$\big(\Omega^{1}(B)/d\,\Omega^{0}(B)\big)\oplus
	\,C^{\infty}(P,\mathfrak{g}^{*})^{G}\,.$
\begin{lemma}\label{lemme frechet pour le quotient}
	The space $d\,\Omega^{0}(B)$ is closed in $\Omega^{1}(B)\,.$ 
	In particular, the quotient $\Omega^{1}(B)/d\,\Omega^{0}(B)$ 
	is a Fr\'echet space.
\end{lemma}
\begin{proof}
	Let $(f_{n})_{n\in \mathbb{N}}$ be a sequence of 
	$\Omega^{0}(B)$ such that $df_{n}\rightarrow \alpha 
	\in \Omega^{1}(B)$ for $\alpha \in \Omega^{1}(B)\,.$ 
	We have to show that the form $\alpha$ is exact. For that, it is sufficient 
	to show that the integral of $\alpha$ on any smooth closed 
	curve of $B$ is zero.\\ Let $c\,:\,S^{1}\rightarrow B$ be a smooth 
	closed curve of $B\,.$ From the continuity of integration on 
	$\Omega^{1}(B)\,,$ it follows that :
	\begin{eqnarray}
		\int_{c}\,\alpha=\int_{c}\,\underset{n\rightarrow\infty}{\text{lim}}(df_{n})=
		\underset{n\rightarrow\infty}{\text{lim}}\,\int_{c}\,df_{n}=0\,.\nonumber
	\end{eqnarray}
	This proves the lemma.
\end{proof}
	The set $C^{\infty}(P,\mathfrak{g}^{*})^{G}$ being closed in the 
	Fr\'echet space $C^{\infty}(P,\mathfrak{g}^{*})\,,$ it follows 
	that $C^{\infty}(P,\mathfrak{g}^{*})^{G}$ is naturally a Fr\'echet space and, 
	according to Lemma \ref{lemme frechet pour le quotient}, 
	the direct sum $\big(\Omega^{1}(B)/d\,\Omega^{0}(B)\big)\oplus
	\,C^{\infty}(P,\mathfrak{g}^{*})^{G}$ is a Fr\'echet space.
\begin{remarque}
	Lemma \ref{lemme frechet pour le quotient} implies that the 
	sum \eqref{equation hodge} is a topological sum.
\end{remarque}
\begin{proof}[Proof of Proposition \ref{proposition dual regulier}]
	We will explicitly construct an inverse of $\Psi\circ A\,.$ 
	First, observe that the relations $\mathfrak{X}(B,V\mu^{B})=(1/V)\,
	\mathfrak{X}(B,\mu^{B})$ and $\mathfrak{X}(B)=\mathfrak{X}(B,\mu^{B})
	\oplus V\nabla \Omega^{0}(B)$ (see Lemma \ref{lemme hodge ameliore}) 
	imply the decomposition $\mathfrak{X}(B)=\mathfrak{X}(B,V\mu^{B})\oplus 
	\nabla \Omega^{0}(B)\,.$ With respect to this decomposition, 
	we define $\mathbb{P}\,:\,\mathfrak{X}(B)\rightarrow\mathfrak{X}
	(B,V\mu^{B})\,,$ the associated projection. One can check that 
	the map $\Big(\Omega^{1}(B)/d\Omega^{0}(B)\Big)\oplus 
	C^{\infty}(P,\mathfrak{g}^{*})^{G}\rightarrow \mathfrak{X}(B,V\mu^{B})
	\oplus C^{\infty}(P,\mathfrak{g})^{G}\,,([\alpha],\zeta)\mapsto 
	(\mathbb{P}(\alpha^{\sharp}),\zeta^{\sharp})$ is the inverse of 
	$\Psi\circ A$ (``\,$\text{}^{\sharp}\,"$ denotes the inverse of 
	the dualisation operator ``\,$\text{}^{\flat}\,"$)\,.\\
	For the continuity of $\Psi\circ A$ and its inverse, one can use 
	arguments similar to those we used in 
	Lemma \ref{lemme la continuite de l isomorphsime}\,.
\end{proof}
\begin{remarque}\label{remarque identification et accouplement}
	Since the two vector spaces $\Big(\mathfrak{X}(B,V \mu^{B})\oplus 
	C^{\infty}(P,\mathfrak{g})^{G}\Big)_{\text{reg}}^{*}$ and 
	$\Big(\Omega^{1}(B)/d\,\Omega^{0}(B)\Big)\oplus\,C^{\infty}
	(P,\mathfrak{g}^{*})^{G}$ are linearly isomorphic via 
	$\Psi\,,$ it follows that the spaces $\mathfrak{X}(B,V \mu^{B})\oplus C^{\infty}
	(P,\mathfrak{g})^{G}$ and $\Big(\Omega^{1}(B)/d\,\Omega^{0}(B)\Big)
	\oplus\,C^{\infty}(P,\mathfrak{g}^{*})^{G}$ are naturally in 
	duality, the pairing, according to 
	\eqref{equation l'accouplement bestiale}, being :
	\begin{eqnarray}
		\Big(([\alpha],\xi),\,(X,f)\Big):=
		\int_{B}\,\alpha(X)\cdot V\mu^{B}+\int_{P}\,(\xi,f)\cdot\mu^{P}\,,
	\end{eqnarray}
	for $\alpha\in \Omega^{1}(B)\,,$ $\xi\in C^{\infty}
	(P,\mathfrak{g}^{*})^{G}\,,$ $X\in \mathfrak{X}(B,V \mu^{B})$ 
	and $f\in C^{\infty}(P,\mathfrak{g})^{G}\,.$
\end{remarque}
\subsection{Determination of the Euler equation}\label{puoette puoette!!}
	With the above identifications of Fr\'echet spaces, namely  
	$\text{}$\,$\mathfrak{X}(P,\mu^{P})^{G}\cong\mathfrak{X}(B,V \mu^{B})
	\oplus C^{\infty}(P,\mathfrak{g})^{G}$ and $\big(\mathfrak{X}
	(P,\mu^{P})^{G}\big)_{reg}^{*}$ 
	$\cong\Big(\Omega^{1}(B)/d\,\Omega^{0}(B)\Big)\oplus\,
	C^{\infty}(P,\mathfrak{g}^{*})^{G}\,,$ we can give a geometrical 
	description of the map $ad^{*}$ associated to the Lie 
	algebra $\mathfrak{X}(P,\mu^{P})^{G}\,.$ 
\begin{proposition}\label{proposition formule adjoint}
	For $X\in \mathfrak{X}(B,V\mu^{B})\,,$ $f\in C^{\infty}
	(P,\mathfrak{g})^{G}\,,$ $\alpha\in \Omega^{1}(B)$ and 
	$\xi\in C^{\infty}(P,\mathfrak{g}^{*})^{G}\,,$ we have
	\begin{eqnarray}
		\text{ad}^{\,*}(X,f)([\alpha],\xi)=
		\Big(\big[-\mathcal{L}_{X}\alpha-\widetilde{(\xi,df)}+
		\widetilde{(\xi,i_{X^{*}}\Omega)}\big]\,,\,
		-\text{ad}^{\,*}(f)\,\xi-X^{*}(\xi)\Big)\,,\label{equation formule adjoint}
	\end{eqnarray}
	where $\widetilde{(\xi,df)}$ and $\widetilde{(\xi,i_{X^{*}}\Omega)}$ 
	are two 1-forms of $B$ defined for $b\in B\,,$ $Z\in \mathfrak{X}(B)$ 
	and $x\in P$ such that $\pi(x)=b\,,$ by :
	\begin{eqnarray}
		\widetilde{(\xi,df)}_{b}(Z_{b}):=\Big(\xi(x),\,(df)_{x}Z^{*}_{x}\Big)
		\label{equation def forme 1}\,\,\,\,\,\,\,\,\,\,\,
		\text{and}\,\,\,\,\,\,\,\,\,\,\,
		\widetilde{(\xi,i_{X^{*}}\Omega)}_{b}(Z_{b}):=
		\Big(\xi(x),\,\Omega(X^{*}_{x},Z^{*}_{x})\Big)\label{equation def forme 2}\,.
	\end{eqnarray}
\end{proposition}
\begin{remarque}
	One can check that the forms defined in \eqref{equation def forme 1} are well defined.
\end{remarque}
\begin{proof}[Proof of Proposition \ref{proposition formule adjoint}]
	Let $X,X'\in \mathfrak{X}(B,V\mu^{B})$ be vector fields with zero 
	divergence on $B\,,$ $f,f'\in C^{\infty}(P,\mathfrak{g})^{G}\,,$ 
	$\alpha\in \Omega^{1}(B)$ and $\xi\in C^{\infty}(P,\mathfrak{g}^{*})^{G}\,.$ 
	From \eqref{equation le crochet de lie}, 
	\eqref{equation l'accouplement bestiale} and Remark 
	\ref{remarque identification et accouplement}, we have :
	\begin{eqnarray*}
		&&\Big(\text{ad}^{\,*}(X,f)([\alpha],\xi),\,(X',f')\Big)=
			-\Big(([\alpha],\xi),\,\text{ad}(X,f)(X',f')\Big)\\
		&&=\bigg(([\alpha],\xi),\,\Big([X,X'],\,[f,f']+
			X^{*}(f')-(X')^{*}(f)+\Omega(X^{*},(X')^{*})\Big)\bigg)\\
		&&=\int_{B}\,\alpha([X,X'])\cdot V\mu^{B}+
			\int_{P}\,\Big(\xi,\,[f,f']+X^{*}(f')-(X')^{*}(f)+
			\Omega(X^{*},(X')^{*})\Big)\cdot\mu^{P}\,.\\
	\end{eqnarray*}
	We now compute separately each term :
	\begin{eqnarray}
		\bullet&&\,\,\int_{B}\,\alpha([X,X'])\cdot V\mu^{B}=
			\int_{B}\,\alpha\wedge i_{[X,X']}(V\mu^{B})=
			\int_{B}\,\alpha\wedge[\mathcal{L}_{X},i_{X'}](V\mu^{B})\nonumber\\
		&&=\int_{B}\,\alpha\wedge\mathcal{L}_{X}\,i_{X'}(V\mu^{B})-
			\int_{B}\,\alpha\wedge i_{X'}
			\underbrace{\mathcal{L}_{X}(V\mu^{B})}_{=0}=
			-\int_{B}\,(\mathcal{L}_{X}\alpha)(X')\cdot V\mu^{B}\,.\\
		\bullet&&\int_{P}\,(\xi,[f,f'])\cdot\mu^{P}=
			\int_{P}\,(\xi,\text{ad}(f)(f'))\cdot\mu^{P}=
			-\int_{P}\,(\text{ad}^{\,*}(f)\,\xi,f')\cdot\mu^{P}\,.\\
		\bullet&&\int_{P}\,(\xi,X^{*}(f'))\cdot\mu^{P}=
			\int_{P}\,X^{*}(\xi,f')\cdot\mu^{P}
			-\int_{P}\,(X^{*}(\xi),f')\cdot\mu^{P}=-\int_{P}\,(X^{*}(\xi),f')\cdot\mu^{P}\,.\\
		\bullet&&-\int_{P}\,(\xi,(X')^{*}(f))\cdot\mu^{P}=
			-\int_{P}\,(\xi,df\big((X')^{*})\big)\cdot\mu^{P}=
			-\int_{B}\,\widetilde{(\xi,df)}(X')\cdot V\mu^{B}\,.\\
		\bullet&&\int_{P}\,\Big(\xi,\Omega(X^{*},(X')^{*})\Big)\cdot\mu^{P}=
			\int_{P}\,\Big(\xi,(i_{X^{*}}\Omega)((X')^{*})\Big)\cdot\mu^{P}=
			\int_{B}\,\widetilde{(\xi,i_{X^{*}}\Omega)}(X')\cdot V\mu^{B}\,.
	\end{eqnarray}
	Hence,
	\begin{eqnarray*}
		\Big(\text{ad}^{\,*}(X,f)([\alpha],\xi),\,(X',f')\Big)=
			\int_{B}\,\Big(-\mathcal{L}_{X}\alpha-
			\widetilde{(\xi,df)}+\widetilde{(\xi,i_{X^{*}}\Omega)}\Big)(X')
			\cdot V\mu^{B}\,\,\,\,\,\,\,\,\,\,\,\,\,\,\,\,\,\,\,\,\,\,\,\,
			\,\,\,\,\,\,\,\,\,\,\,\,\,\,\,\,\,\,\,\,\,
			\,\,\,\,\,\,\,\,\,\,\,\text{}\\
		+\int_{P}\,\Big(-\text{ad}^{\,*}(f)\,\xi-X^{*}
			(\xi),f'\Big)\cdot\mu^{P}
			=\bigg(\Big(\big[{-}\mathcal{L}_{X}\alpha{-}\widetilde{(\xi,df)}{+}
			\widetilde{(\xi,i_{X^{*}}\Omega)}\big],
			-\text{ad}^{\,*}(f)\,\xi{-}X^{*}(\xi)\Big),(X',f')\bigg)\,.\,\,\,\,\,\text{}
	\end{eqnarray*}
	The proposition follows.
\end{proof}
\begin{theoreme}\label{theorem equation d'euler}
	The Euler equation of the group $\text{SAut}(P,\mu^{P})$ on 
	the regular dual of $\mathfrak{X}(P,\mu^{P})^{G}$ can be written :
	\begin{eqnarray}\label{equations d'euler enfin!!!!!!!!}
		\left\lbrace
		\begin{array}{ccc}
			\dfrac{d}{dt}[\alpha]&=&\big[-\mathcal{L}_{X}\alpha-
			\widetilde{(\xi,df)}+\widetilde{(\xi,i_{X^{*}}\Omega)}\big]\,\,,\\
			\dfrac{d}{dt}\xi&=&-\text{ad}^{\,*}(f)\,\xi-X^{*}(\xi)\,,
		\end{array}
		\right.
	\end{eqnarray}
	where $X\in\mathfrak{X}(B,V\mu^{B})\,,$ $\alpha\in 
	\Omega^{1}(B)\,,$ $f\in C^{\infty}(P,\mathfrak{g})^{G}\,,$ 
	$\xi\in C^{\infty}(P,\mathfrak{g}^{*})^{G}$ 
	(these quantities being time-dependant) and where 
	$$
		\left\lbrace
		\begin{array}{cc}
			f^{\flat}=\xi\,,\,\,\text{i.e.,}\,\,\,\xi(x):=
			h^{\mathfrak{g}}_{x}(f(x),\,.\,)\,\,\,\text{for}\,\,\,x\in P\,;\\
			\text{}[X^{\flat}]=[\alpha]\,\,\,\text{where}
			\,\,\,X^{\flat}_{x}:=h^{B}_{x}(X_{x},\,.\,)\,\,\,\text{for}\,\,\,x\in B\,.
		\end{array}
		\right.
	$$
\end{theoreme}
\begin{remarque}
	According to Remark \ref{remarque sur le lien EUler geodesiques}, 
	equations \eqref{equations d'euler enfin!!!!!!!!} describes --at 
	least formally-- geodesics in 
	$\textup{SDiff}(P,\mu^{P})$ with respect to the natural $L^{2}$-metric; 
	a smooth curve $\varphi$ in $\textup{SDiff}(P,\mu^{P})$ is (formally) a geodesic in 
	$\textup{SDiff}(P,\mu^{P})$ if and only if the curve
	\begin{eqnarray}
		(\Psi\circ A \circ\Phi)\big((R_{\varphi^{-1}})_{*_{\varphi}}\,\dot{\varphi}\big)=
		(\Psi\circ A \circ \Phi)\big(\dot{\varphi}\circ\varphi^{-1}\big)
	\end{eqnarray}
	is a solution of equation \eqref{equations d'euler enfin!!!!!!!!} (see 
	\eqref{definition phi} and proposition \ref{proposition dual regulier} 
	for the definitions of $\Psi\,,A\,,\Phi$).
\end{remarque}
\begin{remarque}
	If the Euclidean structure $h^{\mathfrak{g}}$ on 
	$P\times \mathfrak{g}$ is constant (i.e. independent of the fibers), then :
	\begin{description}
		\item[$\bullet$] $\widetilde{(\xi,df)}=\dfrac{1}{2} 
			d\Big(\Vert f\Vert^{2}\Big)\,,$ thus 
			$[\widetilde{(\xi,df)}]=0\,,$
		\item[$\bullet$] the function $V\in C^{\infty}(B,\mathbb{R}_{+}^{*})$ 
			is constant and $\mathfrak{X}(B,V\mu^{B})=
			\mathfrak{X}(B,\mu^{B})\,.$
	\end{description}
\end{remarque}
\begin{remarque}
	If the Euclidean structure $h^{\mathfrak{g}}$ on 
	$P\times \mathfrak{g}$ is constant and if the curvature 
	$\Omega$ of the bundle $G\hookrightarrow P\rightarrow B$ 
	vanishes, then the first equation of 
	(\ref{equations d'euler enfin!!!!!!!!}) reduces to the autonomous equation :
	\begin{eqnarray}
		\dfrac{d}{dt}[\alpha]&=&\big[-\mathcal{L}_{X}\alpha\big]\,.
	\end{eqnarray}
	In this case, system (\ref{equations d'euler enfin!!!!!!!!}) 
	models the passive motion
	in ideal hydrodynamical flow (see \cite{Vizman-geodesics}, \cite{Hattori})\,.
\end{remarque}
\begin{remarque}
	Using the formula $\mathcal{L}_{X}(X^{\flat})=
	(\nabla_{X}X)^{\flat}+\frac{1}{2}d(h^{B}(X,X))$ for 
	$X\in \mathfrak{X}(B)$ (see \cite{Arnold-Khesin}), 
	we can rewrite the first equation of 
	(\ref{equations d'euler enfin!!!!!!!!}) as :
	\begin{eqnarray}
		\dfrac{d}{dt} X=-\nabla_{X}X-\widetilde{(\xi,df)}^{\sharp}+
		\widetilde{(\xi,i_{X^{*}}\Omega)}^{\sharp}+\nabla p\,,
	\end{eqnarray}
	where $p\in C^{\infty}(B,\mathbb{R})$ is determined by the 
	condition $\text{div}_{\,V\mu^{B}}(X)=0\,.$
\end{remarque}
	If we specialize to the case of a $S^{1}$-principal bundle 
	with a 3-dimensional base manifold, and if $h^{\mathfrak{g}}$ 
	is given by the formula $h^{\mathfrak{g}}_{x}(\rho,\varrho):=\rho\varrho$
	for $x\in P$ and $\rho,\varrho\in \mathbb{R}$ (we identify the 
	Lie algebra of $S^{1}$ with $\mathbb{R}$), then :
	\begin{description}
		\item[$\bullet$] the curvature $\Omega$ projects itself on a 2-form 
			$\widetilde{\Omega}\in \Omega^{2}(B)\,.$ 
			Similarly, any function 
			$f\in C^{\infty}(P,\mathbb{R})^{S^{1}}$ projects itself on a function 
			$\tilde{f}\in C^{\infty}(B,\mathbb{R})\,.$
		\item[$\bullet$] One can define a vector field 
			$\mathfrak{B}\in \mathfrak{X}(B,\mu^{B})$ via the 
			relation $i_{\mathfrak{B}}\,\mu^{B}=\widetilde{\Omega}\,,$ 
		\item[$\bullet$] we have the formula $X\times \mathfrak{B}=
			\big(i_{X}\,\widetilde{\Omega}\big)^{\sharp}$ for all vector 
			fields $X\in \mathfrak{X}(B)\,.$
	\end{description}
	In these conditions, it is easy to see that 
	(\ref{equations d'euler enfin!!!!!!!!}) is equivalent to :
	\begin{eqnarray}\label{equation magneto hydo...}
		\left\lbrace
		\begin{array}{ccc}
			X=-\nabla_{X}X+\tilde{f}\,X\times \mathfrak{B}+\nabla p\,,\\
			\dfrac{d}{dt}\tilde{f}=-X(\tilde{f})\,.
		\end{array}
		\right.
	\end{eqnarray}
	These equations, known as the ``superconductivity
	equations'', models the motion of an ideal charged fluid in a given 
	magnetic field $\mathfrak{B}$ where $X$
	represents the velocity field and $\tilde{f}$ the charge 
	density (see \cite{Vizman})\,.
\begin{remarque}
	The appearance of the magnetic term $\mathfrak{B}$ in 
	\eqref{equation magneto hydo...} is not surprising since classical 
	electromagnetism is described in the language of gauge theories, 
	where electromagnetic field is interpreted as the curvature of a 
	connection form on a $S^{1}-$principal bundle. 
\end{remarque}
\begin{remarque}
	If the Euclidean structure $h^{\mathfrak{g}}$ on 
	$P\times \mathfrak{g}$ is constant, then the metric $h^{P}$ 
	turns out to be a Kaluza-Klein metric on $P$ 
	(see formula (2.5) of \cite{Gay-Balmaz-Ratiu}) and 
	\eqref{equations d'euler enfin!!!!!!!!} becomes a particular case of 
	the Euler-Yang-Mills equations of an incompressible homogeneous 
	Yang-Mills ideal fluid (compare with formula (5.23) in 
	\cite{Gay-Balmaz-Ratiu}). The absence of an electric term in 
	\eqref{equations d'euler enfin!!!!!!!!} seems to be due to the fact 
	that the connection $\theta$ is not a dynamical variable in our 
	framework. This is not surprising since in the Yang-Mills 
	formulation of electromagnetism, the configuration space is the 
	space of all connections of the principal bundle describing the physical system.    
\end{remarque}

\section{The group $\text{SAut}\,(P,\mu^{P})$ as the total
	space of a $\text{Gau}(P)$-principal bundle}
	\label{chapitre 2 partie 4}
\subsection{The principal fiber bundle structure of 
	$\text{SAut}\,(P,\mu^{P})$}\label{chapitre 2 section 2.4.1}
	For $\varphi\in \text{Aut}(P)\,,$ we denote by 
	$\widetilde{\varphi}\in \text{Diff}(B)$ the unique 
	diffeomorphism of $B$ satisfying :
	\begin{eqnarray}\label{equation le projete d'un diffeo}
		 \widetilde{\varphi}\circ\pi=\pi\circ \varphi\,\,.
	\end{eqnarray}
	Note that the map $\overline{p}\,:\,\text{Aut}(P)\rightarrow 
	\text{Diff}(B)\,,\varphi\rightarrow\widetilde{\varphi}$ 
	is a group morphism\,.
\begin{proposition}\label{proposition le secret du fibre principal}
	An automorphism $\varphi\in \text{Aut}(P)$ belongs to 
	$\text{SAut}\,(P,\mu^{P})$ if and only if 
	$\widetilde{\varphi}\in \text{SDiff}\,(B,V\mu^{B})\,.$
\end{proposition}
\begin{proof} 
	From \eqref{equation formule forme volume}\,, we have :
	\begin{eqnarray}
		\varphi^{*}\mu^{P}&=&\varphi^{*}\Big((V\circ\pi)\cdot\pi^{*}
		\mu^{B}\wedge\theta^{*}\,\nu_{e}^{G}\Big)=
		(V\circ\pi\circ\varphi)\cdot\Big(\varphi^{*}
		\pi^{*}\mu^{B}\Big)\wedge\Big(\varphi^{*}\theta^{*}
		\nu^{G}_{e}\Big)\,.
	\end{eqnarray}
	For $\varphi\in \text{Aut}(P)\,,$ we write 
	$f^{\varphi}\in C^{\infty}(B,\mathbb{R}^{*})$ the unique 
	function determined by the relation 
	$\widetilde{\varphi}^{*}\mu^{B}=f^{\varphi}\cdot\mu^{B}\,.$ We then have :
	\begin{eqnarray}
		(V\circ\pi\circ\varphi)\cdot\Big(\varphi^{*}\pi^{*}\mu^{B}\Big)&=&
			(V\circ\widetilde{\varphi}\circ\pi)\cdot\big(\pi\circ\varphi\big)^{*}
			\mu^{B}=(V\circ\widetilde{\varphi}\circ\pi)\cdot
			\big(\widetilde{\varphi}\circ\pi\big)^{*}\mu^{B}\nonumber\\
		&=&(V\circ\widetilde{\varphi}\circ\pi)
			\cdot\pi^{*}\widetilde{\varphi}^{*}\mu^{B}=
			(V\circ\widetilde{\varphi}\circ\pi)\cdot\pi^{*}
			\big(f^{\varphi}\cdot\mu^{B}\big)\nonumber\\
		&=&(V\circ\widetilde{\varphi}\circ\pi)\cdot 
			(f^{\varphi}\circ \pi)\cdot\pi^{*}\mu^{B}\label{equation lalalalalalalalal}\,.
	\end{eqnarray}
	On the other hand, for $x\in P\,,$ and for vertical tangent 
	vectors $u_{1},...,u_{m}\in T_{x}P$  
	(we assume $\text{dim}(G)=m$), we have :
	\begin{eqnarray}
		\Big(\varphi^{*}(\theta^{*}\nu^{G}_{e})\Big)_{x}(u_{1},...,u_{m})&=&
			(\theta^{*}\nu^{G}_{e})_{\varphi(x)}
			(\varphi_{*_{x}}u_{1},...,\varphi_{*_{x}}u_{m})\nonumber\\
		&=&(\nu_{e}^{G})\Big(\theta_{\varphi(x)}(\varphi_{*_{x}}u_{1}),...,
			\theta_{\varphi(x)}(\varphi_{*_{x}}u_{m})\Big)\nonumber\\
		&=&(\nu_{e}^{G})\Big((\varphi^{*}\theta)_{x}(u_{1}),...,
			(\varphi^{*}\theta)_{x}(u_{m})\Big)\,.\nonumber
	\end{eqnarray}
	The diffeomorphism $\varphi\,$ being $G$-equivariant, one can 
	show that $\varphi^{*}\theta\,$ is a connection form. 
	In particular, $u_{i}$ being vertical, 
	$(\varphi^{*}\theta)_{x}(u_{i})=
	\theta_{x}(u_{i})$ for $i\in \{1,...,m\}\,,$ and also,
	\begin{eqnarray}
		\Big(\varphi^{*}(\theta^{*}\nu^{G}_{e})\Big)_{x}(u_{1},...,u_{m})&=&
			(\nu_{e}^{G})\Big(\theta_{x}(u_{1}),...,
			\theta_{x}(u_{m})\Big)=
			(\theta^{*}\nu^{G}_{e})_{x}(u_{1},...,u_{m})\label{equation dfljdgjojlfh}\,.
	\end{eqnarray}
	From Lemma \ref{lemme pas assez de place}, \eqref{equation dfljdgjojlfh} 
	and \eqref{equation lalalalalalalalal}\,, we get
	\begin{eqnarray*}
		\varphi^{*}\mu^{P}=(V\circ\widetilde{\varphi}\circ \pi)
		\cdot (f^{\varphi}\circ\pi)\cdot\pi^{*}\mu^{B}\wedge
		\theta^{*}\nu_{e}^{G}=\dfrac{V\circ \widetilde{\varphi}
		\circ\pi\cdot f^{\varphi}\circ\pi}{V\circ\pi}\cdot\mu^{P}\,.
	\end{eqnarray*}
	Thus,
	\begin{eqnarray*}
		\varphi^{*}\mu^{P}=\mu^{P}&\Leftrightarrow&
			\dfrac{V\circ \widetilde{\varphi}\circ\pi\cdot 
			f^{\varphi}\circ\pi}{V\circ\pi}=1\,\,\Leftrightarrow\,\, 
			f^{\varphi}\circ\pi=\bigg(\dfrac{V}{V\circ\widetilde{\varphi}}\bigg)
			\circ\pi\,\,\Leftrightarrow 
			\,\,f^{\varphi}=\dfrac{V}{V\circ\widetilde{\varphi}}\\
		&\Leftrightarrow& \widetilde{\varphi}^{*}\mu^{B}=
			\dfrac{V}{V\circ\widetilde{\varphi}}\cdot\mu^{B}\,\,
			\Leftrightarrow\,\,\widetilde{\varphi}^{*}(V\mu^{B})=V\mu^{B}\,.
	\end{eqnarray*}
	This proves the proposition.
\end{proof}
	Before we show that $\text{SAut}(P,\mu^{P})$ is a 
	$\text{Gau}(P)$-principal fiber bundle, where 
	$\text{Gau}(P):=\{\varphi\in \text{Aut}(P)\,\vert\,\tilde{\varphi}=Id_{B}\}\,,$ 
	we will first prove that $\text{Aut}(P)$ is a $\text{Gau}(P)$-principal 
	fiber bundle and we will see how to use Proposition 
	\ref{proposition le secret du fibre principal} to get a similar 
	result for $\text{SAut}(P,\mu^{P})\,.$\\\\
	Let us recall some basic facts about the group $\text{Gau}(P)$ 
	(see \cite{Kriegl-Michor}, \cite{Michor})\,: 
\begin{proposition}[\cite{Michor}]
	We have :
		\begin{description}
			\item[$(i)$] the group $\text{Gau}(P)=\{\varphi\in \text{Aut}(P)\,
				\vert\,\tilde{\varphi}=Id_{B}\}$ is a closed Fr\'echet 
				Lie subgroup of $\text{Aut}(P)$ whose Lie algebra can be identified 
				with the space of vertical vector fields of $P$ (see Theorem 3.1 and 
				Theorem 3.7 in \cite{Michor}),
			\item[$(ii)$] The set $\{f\in C^{\infty}(P,G)\,\vert\,f\circ \vartheta_{g}=
				c_{g^{-1}}\circ f\,,\forall g\in G\}=:C^{\infty}(P,G)^{G}$ (where 
				$c_{g}\,:\,G\rightarrow G, h\rightarrow ghg^{-1}$), is a closed Fr\'echet 
				Lie subgroup of the current group $C^{\infty}(P,G)$ endowed with the 
				pointwise multiplication (see \cite{Pressley-Segal}), whose Lie algebra 
				can be identified with the Fr\'echet space $C^{\infty}(P,\mathfrak{g})^{G}:
				=\{f\in C^{\infty}(P,\mathfrak{g})\,\big\vert\,f\circ \vartheta_{g}=
				Ad(g^{-1})\,f\,\,,\forall g\in G\}\,,$
			\item[$(iii)$] we have an isomorphism of Fr\'echet Lie groups :
				\begin{eqnarray}
					C^{\infty}(P,G)^{G}\rightarrow \text{Gau}(P)\,,\,\,\,f\mapsto 
					\vartheta_{f(\,.\,)}(\,.\,)\,.\label{definition iso gauge cinfty}
				\end{eqnarray}
		\end{description}
\begin{remarque}
	Note that the above proposition is expressed in the category of Fr\'echet Lie groups, and not 
	in the category of tame Fr\'echet Lie groups of Hamilton (see \cite{Michor}). 
	This is not really burdensome since, 
	in the rest of this paper, we will not have to use the inverse function Theorem of Nash-Moser. 
	Consequently, we don't need the subtle category of Hamilton anymore, and the 
	rest of this paper should be --unless otherwise stated-- understood withing the framework of 
	Fr\'echet Lie groups. 
\end{remarque}
\end{proposition}

	In the following, we will often identify $\text{Gau}(P)$ and 
	$C^{\infty}(P,G)^{G}$ via the isomorphism defined in \eqref{definition iso gauge cinfty}\,.
	Let us introduce some terminology :
	\begin{description}
		\item[$\bullet$] let $\lambda\,:\,\text{Aut}(P)\times 
			\text{Gau}(P)\rightarrow \text{Aut}(P)$ be the right 
			action of the group $\text{Gau}(P)$ on $\text{Aut}(P)\,,$ defined by :
			\begin{eqnarray}
				\Big(\lambda(\varphi,f)\Big)(x):=
				\vartheta_{f(x)}\big(\varphi(x)\big)\,,
			\end{eqnarray}
			for $\varphi\in \text{Aut}(P),\,f\in 
			\text{Gau}(P)$ and $x\in P\,,$
		 \item[$\bullet$]  for $X\in \mathfrak{X}(P)^{G}\,,$ 
			let $\widetilde{X}\in \mathfrak{X}(B)$ be the vector 
			field defined by $\widetilde{X}_{b}:=\pi_{*_{x}}X_{x}$ 
			for $b\in B$ and where $x\in P$ is such that $\pi(x)=b\,,$
		 \item[$\bullet$] $\text{Diff}^{\,\sim}(B):=\{\widetilde{\varphi}
			\in \text{Diff}(B)\,\vert\,\varphi\in \text{Aut}(P)\}$\,\, 
			\,(according to \eqref{equation le projete d'un diffeo}, 
			$\text{Diff}^{\,\sim}(B)$ is a group).
	\end{description}
\begin{lemma}\label{lemme composantes connexes}
	The group $\text{Diff}^{\,\sim}(B)$ is a union of connected 
	components of $\text{Diff}\,(B)$ containing 
	$\text{Diff}^{\,0}(B)\,.$ In particular, $\text{Diff}^{\,\sim}(B)$ 
	is naturally a tame Fr\'echet Lie group.
 \end{lemma}
\begin{proof}
	Let $\varphi\in \text{Aut}(P)$ be an automorphism of $P$ and 
	$\psi$ an element of the connected component of 
	$\text{Diff}(B)$ containing $\widetilde{\varphi}\,.$ 
	To prove the lemma, it is sufficient to show that 
	$\psi\in \text{Diff}^{\,\sim}(B)\,.$\\
	Let $\psi_{t}$ be a smooth curve of $\text{Diff}(B)$ joining 
	$\widetilde{\varphi}$ and $\psi\,,$ i.e. :
	\begin{eqnarray*}
		\psi_{0}=\widetilde{\varphi}\,\,\,\,\,
		\text{and}\,\,\,\,\psi_{1}=\psi\,.
	\end{eqnarray*}
	 For $t_{0}\in [0,1]$ and $x_{0}\in B\,,$ we set
	\begin{eqnarray*}
		(X_{t_{0}})_{x_{0}}:=\dfrac{d}{dt}\bigg\vert_{t_{0}}\,\psi_{t}
		\Big(\psi_{t_{0}}^{-1}(x_{0})\Big)\,.
	 \end{eqnarray*}
	It turns out that $X$ is a time-dependant vector field on $B$ with 
	the property that the flow $\varphi_{t}^{X^{*}_{t}}$ of its 
	horizontal lift $X^{*}_{t}$ satisfies :
	\begin{eqnarray*}
		\widetilde{\big(\varphi^{X^{*}_{t}}_{t}\big)}=
		\varphi^{\widetilde{X}^{*}_{t}}_{t}=\varphi^{X_{t}}_{t}=
		\psi_{t}\circ\widetilde{\varphi}^{-1}\,.
	\end{eqnarray*}
	Thus,
	\begin{eqnarray*}
		\widetilde{\big(\varphi^{X^{*}_{t}}_{t}\circ\varphi\big)}=
		\widetilde{\varphi^{X^{*}_{t}}_{t}}\circ
		\widetilde{\varphi}=\psi_{t}\circ\widetilde{\varphi}^{-1}
		\circ\widetilde{\varphi}=\psi_{t}\,.
	\end{eqnarray*}
	It follows that $\psi=\psi_{1}=\widetilde{\big(\varphi^{X^{*}_{1}}_{1}
	\circ\varphi\big)}$ belongs to 
	$\text{Diff}^{\,\sim}(B)\,.$
\end{proof}
\begin{lemma}\label{lemme c'est bien lisse}
	If $\varphi,\psi\in \text{Aut}(P)$ satisfy 
	$\widetilde{\varphi}=\widetilde{\psi}\,,$ then the map
	\begin{eqnarray}\label{equation formule de lamda}
		\Lambda(\varphi,\psi)\,:\,P\rightarrow G\,,\,x\mapsto 
		\big(\vartheta_{\varphi(x)}^{-1}\big)(\psi(x))\,,
	 \end{eqnarray}
	is smooth.
\end{lemma}
\begin{proof}
	Let $U$ and $V$ be the domains of two trivializing charts of $B$
	$$
		\begin{diagram}
			\node{\pi^{-1}(U)} \arrow[2]{e,t}{\displaystyle\Psi_{U}} 
			\arrow{se,b}{\displaystyle\pi} 
			\node[2]{U\times G} \arrow{sw,b}{\displaystyle pr_{1}^{U}}\\
			\node[2]{U}
		\end{diagram}\,\,\,,\,\,\,
		\begin{diagram}
			\node{\pi^{-1}(V)} \arrow[2]{e,t}{\displaystyle\Psi_{V}} 
			\arrow{se,b}{\displaystyle\pi} 
			\node[2]{V\times G} \arrow{sw,b}{\displaystyle pr_{1}^{V}}\\
			\node[2]{V}
		\end{diagram}\,\,\,
	$$
	such that $\widetilde{\varphi}(U)\subseteq V\,.$ 
	As $\varphi$ and $\psi$ are $G$-equivariant, there 
	exists $s^{\varphi},\,s^{\psi}\in C^{\infty}(U,G)$ such that for all 
	$(x,g)\in U\times G\,,$ $(\Psi_{V}\circ\varphi\circ\Psi_{U}^{-1})(x,g)=
	(\widetilde{\varphi}(x),s^{\varphi}(x)\cdot g)$ and 
	$(\Psi_{V}\circ\psi\circ\Psi_{U}^{-1})(x,g)=
	(\widetilde{\varphi}(x),s^{\psi}(x)\cdot g)\,.$
	For $\Psi_{U}^{-1}(x,g)\in \pi^{-1}(U)\,,$ we then have :
	\begin{eqnarray*}
		&&\Big(\Lambda(\varphi,\psi)\circ\Psi_{U}^{-1}\Big)(x,g)=
			\Big(\vartheta^{-1}_{\big(\varphi\circ\Psi_{U}^{-1}\big)(x,g)}\Big)
			\Big((\psi\circ\Psi_{U}^{-1})(x,g)\Big)\\
		&\Rightarrow&\vartheta\bigg(\big(\varphi\circ\Psi_{U}^{-1}\big)(x,g),\,
			\big(\Lambda(\varphi,\psi)\circ\Psi_{U}^{-1}\big)(x,g)\bigg)=
			\Big(\psi\circ\Psi_{U}^{-1}\Big)(x,g)\\
		&\Rightarrow& \vartheta_{\Big(\Lambda(\varphi,\psi)
			\circ\Psi_{U}^{-1}\Big)(x,g)}\,\bigg(\Big(\Psi_{V}
			\circ\varphi\circ\Psi_{U}^{-1}\Big)(x,g)\bigg)=
			\Big(\Psi_{V}\circ\psi\circ\Psi_{U}^{-1}\Big)(x,g)\\
		&\Rightarrow&\vartheta_{\Big(\Lambda(\varphi,\psi)\circ
			\Psi_{U}^{-1}\Big)(x,g)}\,\Big(\widetilde{\varphi}(x),
			\,s^{\varphi}(x)\cdot g\Big)=
			(\widetilde{\varphi}(x),\,s^{\psi}(x)\cdot g)\\
		&\Rightarrow&\Big(\widetilde{\varphi}(x),\,
			s^{\varphi}(x)\cdot g\cdot(\Lambda(\varphi,\psi)
			\circ\Psi_{U}^{-1})(x,g)\Big)=
			(\widetilde{\varphi}(x),\,s^{\psi}(x)\cdot g)\,.
	\end{eqnarray*}
	It follows that $(\Lambda(\varphi,\psi)\circ\Psi_{U}^{-1})(x,g)=
	g^{-1}\cdot s^{\varphi}(x)^{-1}\cdot s^{\psi}(x)\cdot g\,,$ 
	from which it is easy to see that $\Lambda(\varphi,\psi)$ 
	is smooth\,.
\end{proof}
\begin{lemma}
	The action $\lambda\,:\,\text{Aut}(P)\times \text{Gau}(P)\rightarrow 
	\text{Aut}(P)$ is free and for all $\varphi\in 
	\text{Aut}(P)\,,$\,\, $(\overline{p})^{-1}(\overline{p}(\varphi))=
	\mathcal{O}_{\varphi}$ 
	(here $\mathcal{O}_{\varphi}$ denotes the orbit of $\varphi$ 
	for the action $\lambda$)\,.
\end{lemma}
\begin{proof} 
	The freeness of $\lambda$ is obvious.\\
	Let us fix $\phi\in (\overline{p})^{-1}(\overline{p}(\varphi))\,.$ 
	Since $\widetilde{\phi}=\widetilde{\varphi}\,,$ there exists a unique 
	map $f\in C^{\infty}(P,G)$ satisfying $f(x):=\vartheta_{\varphi(x)}^{-1}
	\big(\phi(x)\big)$ for all $x\in P\,.$ According to Lemma 
	\ref{lemme c'est bien lisse}, this map is smooth and one can 
	check that $f\in \text{Gau}(P)$ and also that $\phi=\lambda(\varphi,f)\,.$ 
	Thus, $\phi\in \mathcal{O}_{\varphi}$ and $ (\overline{p})^{-1}
	(\overline{p}(\varphi)) \subseteq\mathcal{O}_{\varphi}\,.$ 
	The inverse inclusion being trivial, the lemma follows.
\end{proof}
\begin{lemma}\label{lemme section locale pppp}
	The map $\text{Aut}(P)\overset{\overline{p}}{\longrightarrow}
	\text{Diff}^{\,\sim}(B)$  is smooth and 
	admits local smooth sections. 
\end{lemma}
\begin{proof}
	The map $\overline{p}$ being a morphism, it is sufficient to show 
	that there exists a local smooth section of 
	$\overline{p}$ in a neighbourhood of $Id_{B}$ in 
	$\text{Diff}(B)\,.$\\
	Recall that $\pi\,:\,(P,h^{P})\rightarrow(B,h^{B})$ is a 
	Riemannian submersion (Lemma \ref{metrique sur P})\,. Therefore, 
	\begin{eqnarray}\label{equation submersion riemannienne}
		\pi\big(\text{exp}(X_{x})\big)=
		\text{exp}_{\pi(x)}\big(\widetilde{X}_{\pi(x)}\big)\,,
	\end{eqnarray}
	for all $X\in \mathfrak{X}(P)^{G}$ and $x\in P\,.$ 
	For a $G$-invariant vector field $X\in \mathfrak{X}(P)^{G}$ 
	sufficiently closed to 0, the map $\varphi\,:\,P\rightarrow P,\,x
	\mapsto\text{exp}_{x}(X_{x})$ is a diffeomorphism of $P$ 
	(observe that $\varphi\in \text{Aut}(P)$ since $X$ and $h^{P}$ 
	are $G-$invariant)\,. In view of \eqref{equation submersion riemannienne}, 
	$\overline{p}(\varphi)$ is simply the map $B\rightarrow B\,,\,x
	\mapsto\text{exp}_{x}(\widetilde{X}_{x})\,.$ It follows that the 
	local expression of $\overline{p}$ in the standard charts of 
	$\text{Aut}(P)$ and $\text{Diff}^{\,\sim}(B)\,,$ are the 
	projection $\mathfrak{X}(P)^{G}\cong \mathfrak{X}(B)\oplus 
	C^{\infty}(P,\mathfrak{g})^{G}\rightarrow \mathfrak{X}(B)$ on the 
	first factor (see \eqref{definition phi}). Hence this map is 
	smooth and admits local sections.
\end{proof}
\begin{theoreme}[\cite{Michor}]\label{theoreme de Michor}
	The group $\text{Aut}(P)$ is an extension of the group 
	$\text{Diff}^{\,\sim}(B)$ by the gauge group $\text{Gau}(P)\,:$  
	\begin{eqnarray}\label{the short exacte sequence with S}
		\{e\}\longrightarrow \text{Gau}(P)\longrightarrow\text{Aut}(P)
		\longrightarrow\text{Diff}^{\,\sim}(B)\longrightarrow\{e\}\,\,.
	\end{eqnarray}
\end{theoreme}
\begin{remarque}
	Theorem \ref{theoreme de Michor} means that the above sequence 
	is a short exact sequence of Lie groups such that 
	$\text{Aut}(P)$ is a $\text{Gau}(P)$-principal bundle over 
	the group $\text{Diff}^{\,\sim}(B)$ (see \cite{Neeb}). 
\end{remarque}
\begin{proof}
	The sequence \eqref{the short exacte sequence with S} is obviously exact.\\
	Let us show that we have a principal bundle. 
	Let $(\Phi(\mathcal{U}),\Phi^{-1})$ be the standard chart of 
	$\text{Diff}^{\,\sim}(B)\,,$ i.e.,  
	$\mathcal{U}\subseteq \mathfrak{X}(B)$ and 
	$\Phi^{-1}$ is defined by $\Phi^{-1}(X)(x):=\text{exp}_{x}(X_{x})\,.$
	We also take a section $\sigma\,:\,\Phi(\mathcal{U})
	\rightarrow \text{Aut}(P)$ of $\overline{p}\,.$ We want 
	to construct a fiber chart of $\text{Aut}(P)$ near the 
	identity using $\Phi(\mathcal{U})\,.$ Let us consider 
	the following diagram :
	$$
		\begin{diagram}
		\node{\overline{p}^{-1}(\Phi(\mathcal{U}))} \arrow[2]{e,t}
		{\displaystyle\Psi_{\Phi(\mathcal{U})}} \arrow{se,b}
		{\displaystyle\overline{p}} \node[2]{\Phi(\mathcal{U})\times 
		\text{Gau}(P)} \arrow{sw,b}{\displaystyle pr_{1}}\\
		\node[2]{\Phi(\mathcal{U}))}
		\end{diagram}
	$$
	where $\Psi_{\Phi(\mathcal{U})}$ is defined by 
	$\Psi_{\Phi(\mathcal{U})}(\varphi):=\Big(\widetilde{\varphi},\,
	\Lambda\big(\sigma(\widetilde{\varphi}),\varphi\big)\Big)$ 
	for $\varphi\in \text{Aut}(P)$ (see Lemma \ref{lemme c'est bien lisse} 
	for the definition of $\Lambda$)\,. Thus it goes to show that 
	$\Psi_{\Phi(\mathcal{U})}$ is :
	\begin{description}
		\item[$\bullet$] smooth according to Lemma 
			\ref{lemme c'est bien lisse} and the characterization of 
			smooth curves in a space of sections,
		\item[$\bullet$] bijective, the inverse being 
			$$
				\Phi(\mathcal{U})\times \text{Gau}(P)\rightarrow 
				(\overline{p})^{-1}(\Phi(\mathcal{U})),\,\,\,\,
				(\chi,f)\mapsto \lambda(\sigma(\chi),\,f)\,,
			$$
		\item[$\bullet$] $\text{Gau}(P)$-equivariant\,.
	\end{description}
	It follows that $\Big((\overline{p})^{-1}(\Phi(\mathcal{U})),\,
	\Psi_{\Phi(\mathcal{U})}\Big)$ is a fiber chart of 
	$\text{Aut}(P)\,,$ and using translations, $\text{Aut}(P)$ 
	becomes a $\text{Gau}(P)$-principal fiber bundle with base space 
	$\text{Diff}^{\,\sim}(B)\,.$
\end{proof}
	We now return to the case of automorphisms of $P$ preserving 
	$\mu^{P}\,.$ Let us set $p\,:\,\text{SAut}(P,\mu^{P})
	\rightarrow\text{SDiff}^{\,\sim}(B,V\mu^{B}):=\{\widetilde{\varphi}\in 
	\text{SDiff}(B,V\mu^{B})\,\vert\,\varphi\in \text{SAut}(P,\mu^{P})\},$ 
	$\,\varphi\rightarrow\widetilde{\varphi}\,.$
\begin{lemma} 
	$\text{}$\,\,\,\,The group $\text{SDiff}^{\,\sim}(B,V\mu^{B})$ is 
	a union of connected components of $\text{SDiff}(B,V\mu^{B})$ 
	containing $\text{SDiff}^{\,0}(B,V\mu^{B})\,.$ 
	In particular, $\text{SDiff}^{\,\sim}(B,V\mu^{B})$ 
	is a tame Fr\'echet Lie group. 
\end{lemma}
\begin{proof} 
	Fix $\varphi\in \text{SAut}(P,\mu^{P})$ and let $\psi$ be an 
	element of the connected component of $\text{SDiff}(B,V\mu^{B})$ 
	containing $p(\varphi)\,.$ As $\text{SDiff}(B,V\mu^{B})\subseteq 
	\text{Diff}(B)\,,$ one may, as in Lemma \ref{lemme composantes connexes}, 
	find $\psi_{1}\in \text{Aut}(P)$ such that $\widetilde{\psi_{1}}=\psi\,.$ 
	But, as $\psi\in \text{SDiff}(B,V\mu^{B})\,,$ Proposition 
	\ref{proposition le secret du fibre principal} implies that 
	$\psi_{1}\in \text{SAut}(P,\mu^{P})$ and thus $\psi\in 
	\text{SDiff}^{\,\sim}(B,V\mu^{B})\,.$
\end{proof}
	Note that the group $\text{Gau}(P)$ also acts on 
	$\text{SAut}(P,\mu^{P})\,:$ for $f\in \text{Gau}(P)$ and 
	$\varphi\in \text{SAut}(P,\mu^{P})\,,$ 
	$$
		\widetilde{\lambda(\varphi,f)}=\widetilde{(\vartheta_{f(\,.\,)}
		\circ \varphi)}=\widetilde{\vartheta_{f(\,.\,)}}\circ
		\widetilde{\varphi}=\widetilde{\varphi}\in \text{SDiff}(B,V\mu^{B})\,,
	$$
	this means, according to Proposition 
	\ref{proposition le secret du fibre principal}, that 
	$\lambda(\varphi,f)\in$ $ \text{SAut}(P,\mu^{P})\,.$ In this context, 
	all the previous analogous lemmas remain valid. 
	For example, existence of local sections of 
	$p\,:\,\text{SAut}(P,\mu^{P})\rightarrow \text{SDiff}^{\,\sim}(B,V\mu^{B})$ 
	can be obtained from Lemma \ref{lemme section locale pppp} 
	(it suffices to take local sections given by Lemma 
	\ref{lemme section locale pppp} and to restrict them to 
	$\text{SDiff}^{\,\sim}(B,V\mu^{B})$)\,. Therefore, 
\begin{theoreme}\label{theoreme de moi!!}
	The Lie group $\text{SAut}(P,\mu^{P})$ is an extension of the 
	Lie group $\text{SDiff}^{\,\sim}(B,V\mu^{B})$ by the gauge 
	group $\text{Gau}(P)\,:$  
	\begin{eqnarray}\label{the short exacte sequence}
		\{e\}\longrightarrow \text{Gau}(P)\longrightarrow
		\text{SAut}(P,\mu^{P})\longrightarrow
		\text{SDiff}^{\,\sim}(B,V\mu^{B})\longrightarrow\{e\}\,\,.
	\end{eqnarray}
\end{theoreme}
\begin{remarque} 
	One may recover Theorem \ref{theorem equation d'euler} 
	from Theorem \ref{theoreme de moi!!} using the description of 
	geodesics on extensions of Lie groups as given in \cite{Vizman}\,.
\end{remarque}
$\text{}$\\
\textbf{{Acknowledgments}}

	I would like to give special 
	thanks to Tilmann Wurzbacher for his careful and critical 
	reading of the ``French version" of this paper (i.e. the corresponding 
	chapter of my thesis). I would also like to thank Tudor Ratiu for 
	having pointed out to me the relation between the Euler equation of 
	the group $\text{SAut}(P,\mu^{P})$ and the Euler-Yang-Mills equations\,.\\
	This work was done with the financial support of the Fonds National 
	Suisse de la Recherche Scientifique under the grant PIO12--120974/1.

\end{document}